\documentclass{amsart} %%%{article}
\usepackage{amssymb}
\usepackage{amsmath}
\usepackage{stmaryrd} %%%Feb 23, 2020 added peng
\usepackage{xcolor}

\newtheorem{theorem}{Theorem}[section]

\newtheorem{lemma}[theorem]{Lemma}

\newtheorem{proposition}[theorem]{Proposition}
\newtheorem{remark}[theorem]{Remark}

\numberwithin{equation}{section}

\begin{document}

\title{Conformal Bach flow}

\author{Jiaqi Chen, Peng Lu, and Jie Qing} 
\date{}

\begin{abstract}
In this article we introduce conformal Bach flow and establish its well-posedness on closed manifolds. 
We also obtain its backward uniqueness. To give an attempt to study the long-time behavior of conformal 
Bach flow, assuming that the curvature and the pressure function are bounded, global and local
Shi's type $L^2$-estimate  of derivatives of curvatures are derived. Furthermore using the 
$L^2$-estimate and based on an idea from \cite{St13} 
we show Shi's pointwise-estimate of derivatives of curvatures
without assuming Sobolev constant bound.
\end{abstract}



\address{Jiaqi Chen, School of Math Sciences, Xiamen University, Xiamen, Fujian 361005, China}
\email{chenjiaqi@xmu.edu.cn}

\address{Peng Lu, Dept of Math, University of Oregon, Eugene, OR 97403, USA}
 \email{penglu@uoregon.edu}
 
\address{Jie Qing, Dept of Math, University of California, Santa Cruz, CA 95064, USA}
\email{qing@ucsc.edu}

\thanks{P.L. is partially supported by Simons Foundation through Collaboration Grant  229727.
J.Q. is partially supported by NSF DMS-1608782.}

\maketitle

%%%%%%%%%%%%%%%%%%%%%%%%%
%%%%%%%%%%%%%%%%%%%%%%%%%
%%%%%%%%%%%%%%%%%%%%%%%%%
%%%%%%%%%%%%%%%%%%%%%%%%%
\section{Introduction}  \label{sect introduction}

It is well known in the field of conformal geometry that it is very desirable to make use of Bach tensor 
(cf. (\ref{eq def Bach tensor}) below) to generate curvature flow. One obvious reason is that Bach tensor 
is the $L^2$-gradient of Weyl curvature in dimension 4. 
Note that some versions of 
Bach flows have been proposed, for example, in \cite[(7.7)]{BB10}. In 2011 Bahuaud and Helliwell 
(\cite[\S 5]{BH11}) found a well-posed curvature flow by Bach tensor as follows: 
 \begin{equation} \label{eq Bauquad Hell Bach f}
 \partial_t g = B + \frac{1}{2(n-1)(n-2)} (\Delta S) g,
 \end{equation}
where $B=B(g)$ is the Bach tensor, $S=S_g$ is the scalar curvature, and $\Delta=\Delta_g$ is the 
Laplace-Beltrami operator on a manifold of dimensions $n$.
 The additional term $(\Delta S) g$ with the appropriate coefficient is important
  and needed for the well-posedness of the flow.

%%% \vskip 0.1in
Ricci flow has become the highlight of achievements in geometric analysis since tremendous
 developments of the use of Ricci flow reignited by Perelman in the early 2000 unfolded.
  One thing we would like to mention is conformal Ricci flow (CRF in short), which was proposed
   by Fischer (\cite{Fi04}) as a modified Ricci flow that keeps the scalar curvature as a unchanged 
   constant by adding a pressure term as follows:
\[
\partial_t g  =- 2 \left (\operatorname{Rc}_{g} - \frac{s_0}{n} \right )g  - 2pg.
\]
This resembles the Navier-Stokes equations in fluid dynamics and hopefully it produces more efficient
 approach to find Einstein metrics if possible. Particularly CRF seems to be transverse
  to classes of conformal metrics, which attracts attentions from researchers in conformal geometry 
  (cf. \cite{FM02, Fi04, LQZ14, SZ18, LQZ19}). 

%%%\vskip 0.1in
Inspired by the idea behind CRF, here we propose the following conformal Bach flow
(CBF in short).  We always assume dimension $n \geq 4$ below.
 Suppose that $(M^n,g_0)$ is an $n$-dimensional  Riemann manifold with constant 
scalar curvature $s_0$. CBF is a family of metrics $\{ g(t) \}_{t \in [0, T]}$ 
which satisfies
\begin{equation} \label{eq cbf ist time scalar const}
    \begin{cases}
    \partial_t g = 2(n-2) \left ( B(g) + pg \right ) \quad \text{on}~ M \times [0,T], \\
    S_{g(t)}= s_0 \quad \text{for }~  t \in [0,T],
    \end{cases}
\end{equation}
where $p=p(t)$ is a family of functions on $M$. This is a fourth order evolution equation
of $g(t)$. 
The pressure function $p(t)$ at each time induces the conformal change to keep the scalar 
curvature as a unchanged constant. Notice that, in contrast to \eqref{eq Bauquad Hell Bach f}, 
the term $(\Delta  S)g$ does not make any appearance in \eqref{eq cbf ist time scalar const}. 
Similar to CRF, the system (\ref{eq cbf ist time scalar const}) is equivalent to the following 
coupled weakly-parabolic and elliptic equations:
\begin{equation} \label{eq cbf eq p ellipt}
    \begin{cases}
    \partial_t g = 2(n-2) \left ( B(g) + pg \right ) \quad\text{on}~ M \times [0,T],   \\
\left ( (n-1)\Delta_{g(t)} + s_0 \right ) p(t) = - (n-2)A(g(t)) \cdot B(g(t))+\nabla_{g(t)}^2 
\cdot B(g(t))
 \end{cases}
\end{equation}
 for each  $t \in [0,T]$. Here $A$ is the Schouten tensor, and we adopt the notations 
 $A \cdot B = A_{ij} B_{ij}$, $\nabla^2 \cdot B = \nabla_i \nabla_j  B_{ij}$, 
 and the covariant derivative $\nabla = \nabla_{g(t)}$. 
 In light of Lemma \ref{eq div B is fourth order} below,  $\nabla^2 \cdot B$ is in 
 fact of order 4 instead of 6 in term of derivatives of metrics. Nevertheless CBF 
 \eqref{eq cbf eq p ellipt} is a fourth order geometric flow of $g$. There are indeed growing 
 interests in higher order geometric flows lately, for instance, among others: Calabi flow
 (\cite{Ca85, Ch01}), Yang-Mills flow (\cite{Str94, SST98, HT04, W19}), Willmore flow
 (\cite{KS01, KS02}),  $L^2$-Riemann curvature flow (\cite{St08, St11, St12a, St12b, St13}), 
   and Ambient obstruction flow (\cite{BH11, Lop18}).

%%%%%%
\vskip .1cm
In this article we initiate the study of CBF. Let us describe the main results of this article. 
First we adopt the DeTurck's trick in Lemma \ref{lem the equi before after DeTurck} 
below to eliminate 
the degeneracy from the invariance under diffeomophisms of the system \eqref{eq cbf eq p ellipt}. 
To prove the short-time existence, following the approach used in \cite{LQZ14}, we first solve $p$
from the 
second equation in (\ref{eq cbf eq p ellipt}) and then the existence of a short-time solution to
 the mixed differential-integral evolution equation (\ref{eq decoupled DeTur CBF}) can be
  established by applying a version of Newton's method based on the existence theory of systems of
   linear parabolic differential equations (see Proposition \ref{prop short time sol decoupled DeTurck CBF}).

\begin{theorem} \label{thm main short time exist cbf}
Let $M^n$ be a closed manifold. Assume that $g_0$ is a $C^{4+\alpha}$ Riemannian metric 
on $M$ with constant scalar curvature $S_{g_0} = s_0$ and that the elliptic operator $(n-1)
 \Delta_{g_0} + s_0$ is invertible. Then there exist a unique $C^{4+\alpha, 1
 + \frac {1}{4}\alpha}$-family of Riemannian metrics $g(t)$ and a $C^{2+\alpha, 0}$-family
 of pressure functions $p(t)$ 
  which solve CBF (\ref{eq cbf eq p ellipt}) for some small $T>0$. 
\end{theorem}

The metrics $g(t)$ and the functions $p(t)$ are shown to be smooth when the initial metric 
is smooth (cf. Theorem \ref{thm main cbf better regu}). We also show that the backward uniqueness
 holds for CBF on closed manifolds. Readers are referred to \cite{Ko10} for the history and significance
of the backward uniqueness of geometric flows.

\begin{theorem} \label{thm backward uniqu CBF}
Let $M^n$ be a closed manifold and let $(g(t), p(t))$ and $(\tilde{g}(t), \tilde{p}(t))$, $t \in [0,T]$,
be two smooth solutions of (\ref{eq cbf eq p ellipt}) on $M$. We assume that operator 
$(n-1)\Delta_{g(t)} + s_0 $ is invertible for each $t$. If $g(T) =\tilde{g}(T)$,  then 
$(g(t), p(t)) =(\tilde{g}(t), \tilde{p}(t))$ for all $t$.
\end{theorem}  

For the study of large time behaviors of solutions to CBF, we follow the approach from \cite{St08, St11, 
St12a, St12b, St13} and establish the integral version of Shi's type estimate of derivatives of  
curvature (cf. Theorem \ref{thm int Shi curvature derivative est cbf}) assuming that the curvature 
and the pressure function are uniformly bounded in space and time. Moreover, a local version of
 integral Shi's type estimate (Theorem \ref{thm improved Shi curvature derivative conseq}) is 
 also established. Finally, based on an argument in the proof of \cite[Theorem 1.3]{St13}, 
 we obtain the  Shi's pointwise-estimate assuming only the  boundedness of  curvature and pressure function
  (without uniform Sobolev constant). 

\begin{theorem}\label{thm ptwise Shi est for CBF-intro} 
Let $M^n$ be a complete manifold and let $(g(t), p(t))$ be a smooth solution to (\ref{eq cbf eq p ellipt})
 in $M \times [0, T]$ with constant scalar curvature $s_0$. 
 We assume that for some constant $K>0$
\begin{equation} \label{eq curv bdd improv ptwise shi assump}
\sup_{(x,t) \in M \times [0,T]} \left ( |\operatorname{Rm}(x, t)| +|p(x,t)| \right ) \leq  K.
\end{equation}
Then, for any $m \in \mathbb{N}$ there exists a constant $C = C(n, m)$ such that 
\begin{equation*}
    |\nabla_{g(t)}^m \operatorname{Rm}(x, t)|_{g(t)} \leq C\left ( K + t^{-\frac{1}{2}}
    \right )^{1+ \frac{m}{2}} 
  ~~~  \text{ for } (x,t) \in M \times (0,T].
\end{equation*}
\end{theorem}

Theorem \ref{thm ptwise Shi est for CBF-intro} provides a criterion of the finite singular time of CBF
 and possibly facilitates further studies of long-time behaviors of CBF.

%%%%
\vskip .1cm
Now we give an outline of the article. In \S \ref{sect short time cbf} we introduce CBF and prove 
Theorem \ref{thm main short time exist cbf} using Newton's method, near the end of the section
 we give a quick proof of the uniqueness and the improved regularity.
In \S \ref{sec backward unique CBF} we prove the backward uniqueness
 (Theroem \ref{thm backward uniqu CBF}). 
In \S \ref{sect Shi conformal Bach f} we prove a  Shi's type $L^2$-estimate
 of derivatives of curvatures for CBF. We also give two quick consequences. One is a characterization of
  when a singularity develops in CBF, and the other is a compactness theorem for a family of CBF. 
  In \S \ref{sect Improved Shi conf Bach flow} we first prove a local version of Shi's type $L^2$-estimate 
  of derivatives of curvatures for CBF, then we  prove a Shi's type pointwise-estimate of derivatives of
   curvatures without assuming uniform Sobolev constant bound 
(Theorem \ref{thm ptwise Shi est for CBF-intro}).

%%%%%%%%%%%%%%%%%%%%%%%%%%%%%%%%%
%%%%%%%%%%%%%%%%%%%%%%%%%%%%%%%%%
%%%%%%%%%%%%%%%%%%%%%%%%%%%%%%%%%
\section{Well-posedness of CBF: existence, uniqueness and regularity}  \label{sect short time cbf}

In this section we prove Theorem \ref{thm main short time exist cbf}. This is done in 
steps. The first step consists of a sequence of conversions. 
We convert solving (\ref{eq cbf eq p ellipt})
into solving modified CBF (\ref{eq cbf ellp Baqua and Helli}),
then we use the DeTurck's trick to further convert the solving into solving
 DeTurck modified CBF system (\ref{eq mod cbf DeTurck flow}) 
 (Lemma \ref{lem the equi before after DeTurck}),
  the last conversion is to use an idea in \cite{LQZ14}
to reduce to solve the decoupled DeTurck modified CBF (\ref{eq decoupled DeTur CBF}).
The second step is to solve the linearized equation of (\ref{eq decoupled DeTur CBF}) 
by  contraction mapping theorem from the solvability of linear parabolic  system 
(see Lemma \ref{lem sola standard parabolic linear system high order})
in \S \ref{subsect lineard DeTurck mod-cbf}.
In the last step we solve (\ref{eq decoupled DeTur CBF}) by Newton's method
in \S \ref{subsect short time decouple DeTurck cbf} 
(Proposition \ref{prop short time sol decoupled DeTurck CBF}).

In the last two subsections we prove the uniqueness and the
 improved regularity. We begin  this section with some preliminaries of CBF.

%%%%%%%%%%%%%%%%%%%%%%%%%%%
\subsection{Conformal Bach flow} \label{subsect cbf}
Let $(M^n, g)$ be a Riemannian manifold.  The Schouten tensor $A =A(g)$ is defined by
\begin{equation} \label{eq def Shouten tensor}
    A_{ij} (g) \doteqdot \frac {1}{n-2} \left(R_{ij} - \frac{S}{2(n-1)} g_{ij}\right),
\end{equation} 
where $R_{ij}$ is the Ricci tensor and $S$ is the scalar curvature of metric $g$. 
Let $W$ be the Weyl tensor.
The Cotton tensor is defined by $C_{ijk} \doteqdot \nabla_k A_{ij} - \nabla_j A_{ik}
 = \frac{n-2}{n-3} \nabla_l W_{lijk}$. 
  The Bach tensor, which is well-known in general relativity, is defined by
\begin{equation}\label{eq def Bach tensor}
    B_{ij} (g)  \doteqdot \frac{1}{n - 3} \nabla_k \nabla_l W_{iklj} + \frac{1}{n - 2}
    R_{kl}W_{iklj} =\nabla_k \nabla_k A_{ij} - \nabla_k \nabla_i A_{jk} +A_{kl}W_{iklj}.
\end{equation}
Metric $g$ is called  Bach flat  if $B(g) =0$.
Note that we adopt the convention that Ricci tensor $R_{ij} =g_{kl}R_{iklj}$. 

Now we list a few basic facts about Bach tensor which will be used later. 
Bach tensor is trace-free. 
In dimension $4$ Bach tensor is conformally invariant.  We have (\cite[Lemma 5.1]{CC13}) 
\begin{lemma} \label{eq div B is fourth order}
For  Riemannian manifold $(M^n, g)$, the divergence of Bach tensor
\begin{equation}
    \nabla_j B_{ij} = - \frac{n-4}{(n-2)^2} \, C_{jki}R_{jk}.
\end{equation}
Hence $\nabla^2 \cdot B(g)= \nabla_i \nabla_j  B_{ij}(g)$ is a differential operator 
of $g$ up to fourth derivatives.
In particular, in dimension 4 Bach tensor is divergence-free.
\end{lemma}

Examples of 4 dimensional Bach flat 
metrics are: (i) those locally conformal to Einstein metrics, and (ii) those half-conformally 
flat (\cite[Proposition 4.78]{Be87}).

%%%%%%%
\vskip .1cm
Now we give a proof of the equivalence of  systems (\ref{eq cbf ist time scalar const}) and 
(\ref{eq cbf eq p ellipt}).
\begin{lemma} \label{lem CBF equiv forms}
Let $M^n$ be a closed manifold. We define a modified CBF by 
\begin{equation} \label{eq cbf ellp Baqua and Helli}
    \begin{cases}
        & \partial_t g = 2(n-2) \left ( B(g) + \frac{1}{2(n-1)(n-2)} (\Delta S_g) g +pg \right ), \\
        & \left ( (n-1)\Delta_{g(t)} +s_0 \right )p(t)= - (n - 2) A(g(t)) \cdot B(g(t)) + 
        \nabla^2_{g(t)} \cdot B(g(t)).
    \end{cases}
\end{equation}
Then the flows defined by  (\ref{eq cbf ist time scalar const}),
(\ref{eq cbf eq p ellipt}),  and (\ref{eq cbf ellp Baqua and Helli}) with initial metric $g(0) =g_0$ 
of constant scalar curvature $s_0$, are all equivalent.
\end{lemma}

\begin{proof}
We prove the equivalences through three parts. Part 1. Suppose that $g(t)$ and the associated 
$p(t)$ is a solution of  (\ref{eq cbf ist time scalar const}), 
we prove that the pair $(g(t), p(t))$ is a solution of (\ref{eq cbf eq p ellipt}) 
and (\ref{eq cbf ellp Baqua and Helli}) by showing that $p(t)$ satisfies 
the second equation in (\ref{eq cbf eq p ellipt}). 
As usual, in the computation below $\Delta, \nabla, R_{ij}$, et al. are defined by metric $g(t)$. 

We apply the first variation formula of scalar curvature (see, for example, \cite[p.109]{CLN6} 
using $v_{ij}(t) = 2(n-2) \left ( B_{ij}+ pg_{ij} \right )$ and $V(t) = g_{ij}(t) v_{ij}(t)$).
 We simplify to get
\begin{align}
\partial_t S_{g(t)} 
&= 2(n - 2) \left [ - \Delta \left ( g_{ij}(B_{ij} + pg_{ij}) \right )+ \nabla_i \nabla_j \left ( B_{ij}
 + pg_{ij} \right ) -R_{ij} \left ( B_{ij} + pg_{ij} \right ) \right ] 
\label{eq S(t) deriv formu} \\
& = 2(n- 2) \left [ - (n-1) \Delta p - s_0 p - (n-2)A_{ij}B_{ij} + \nabla_i \nabla_j
     B_{ij} \right ],  \notag
\end{align}
where we have used the facts that Bach tensor is trace-free and $S_{g(t)}=s_0$,
 to get the last equality.
  Now the  second equation in (\ref{eq cbf eq p ellipt}) follows from $\partial_t S_{g(t)} =0$.
  
%%%%%
\vskip .1cm
Part 2. Suppose that  pair $(g(t), p(t))$ is a solution of (\ref{eq cbf eq p ellipt}), 
we prove that $g(t)$ associated with $p(t)$ is a solution of  (\ref{eq cbf ist time scalar const}) 
and also a solution of (\ref{eq cbf ellp Baqua and Helli}) by showing $S_{g(t)} =s_0$.
 Combining the first equality in (\ref{eq S(t) deriv formu}) and the second equation in
  (\ref{eq cbf eq p ellipt}) we have
\[
\partial_t S_{g(t)} =2(n-2) \left (s_0 - S_{g(t)}  \right ) p (t)
\]
with initial data $S_{g(0)} =s_0$. Hence $S_{g(t)} =s_0$.

%%%%%
\vskip .1cm
Part 3. Suppose that $(g(t), p(t))$ is a solution of (\ref{eq cbf ellp Baqua and Helli}), 
we prove that $(g(t), p(t))$ is a solution of  (\ref{eq cbf ist time scalar const}) by 
showing $S_{g(t)} =s_0$.
Calculating as in Part 1, 
we have
\begin{equation*}
\partial_t S_{g(t)}
=  \left (s_0 - S_{g(t)} \right ) p  - \left ( \Delta^2S_{g(t)}+ \frac{1}{n-1} S_{g(t)} 
     \Delta S_{g(t)} \right )
\label{eq scalar curv evol Baq Hell asso},
\end{equation*} 
where we have used the equations in (\ref{eq cbf ellp Baqua and Helli}) 
to get the equality. 
By the uniqueness of solutions of the 4th order linear (by treating $\frac{1}{n-1} S_{g(t)}$ 
as a known function) parabolic equation on closed manifolds we conclude that $S_{g(t)} =s_0$.
\end{proof}

\begin{remark} \label{rem term with key observ}
In dimension $n = 4$,  since divergence $\nabla_i B_{ij} =0$, (\ref{eq cbf eq p ellipt})
simplifies to
\begin{equation*}
\begin{cases}
 \partial_t g= 4 \left (B+pg \right ),  \\
( 3 \Delta_{g(t)}  + s_0) p(t)  = -2A(g(t)) \cdot B (g(t)). 
\end{cases}
\end{equation*}
Note that also in dimension 4 Bach tensor is the negative of the gradient of the $L^2$-norm functional 
of Weyl tensor (see \cite[p.135]{Be87}), 
we conclude that $\int _M |W_g|_g^2 d \mu_g$  is non-increasing under the flow above.
\end{remark}

%%%%%%%%%%%%%%%%%%%%%%%%%%%
\subsection{DeTurck modified CBF}  \label{subsect DeTurck cbf}
First we give a little motivation for our route of 
using  (\ref{eq cbf ellp Baqua and Helli}) to solve  (\ref{eq cbf eq p ellipt}) described
at the beginning of this section. 
Note that Bach tensor can be written schematically as
\begin{equation}\label{eq Bach definition}
    B_{ij} = \frac{1}{n-2} \Delta R_{ij} -   \frac{1}{2(n-1)(n-2)} (\Delta S) g_{ij} -
     \frac{1}{2(n-1)} \nabla_i \nabla_j S + \text{lower order terms},
\end{equation}
hence we may modify the standard choice of the vector field in the DeTurck's trick for 
Ricci flow by adding the action of $\Delta_g$ (see the first term in (\ref{eq vector field w def})
below) to eliminate the degeneracy in $\frac{1}{n-2} \Delta R_{ij}$ 
caused by the symmetry of diffeomophisms,
 the term  $ -  \frac{1}{2(n-1)(n-2)} (\Delta S) g$ has to be taken care of as in \cite{BH11}
  which partly explains why we use (\ref{eq cbf ellp Baqua and Helli}) rather than
   (\ref{eq cbf eq p ellipt}), and lastly the term $- \frac{1}{2(n-1)} \nabla_i \nabla_j S$ 
   is of Hessian-type and can be taken care off by modifying the choice of the vector field 
    (see the second term in (\ref{eq vector field w def})).

Fixing a metric $\tilde{g}$ with enough smoothness as a background metric and following 
the choice of the vector field for DeTurck's trick as in \cite[\S 5.2]{BH11}, we define vector field
\begin{equation} \label{eq vector field w def}
W_g^k \doteqdot - g^{ij}  \Delta_g  \left ( \Gamma_{ij}^k(g) - \Gamma_{ij}^k(\tilde{g})  
\right ) + \frac{n-2}{2(n-1)} (\nabla_g S_g)^k.
\end{equation}
Recall that Lie derivative  $(\mathcal{L}_W g)_{ij} \doteqdot \nabla_i W_j + \nabla_j W_i$. 
We define the DeTurck modified CBF as
\begin{equation} \label{eq mod cbf DeTurck flow}
    \begin{cases}
\partial_t g = 2(n-2) \left (B(g) + \frac{1}{2(n-1)(n-2)} (\Delta S_g) g + pg \right ) 
             + \mathcal{L}_{W_g}g, \\
\left ( (n-1)\Delta_{g(t)} + s_0 \right ) p(t) =- (n-2) A(g(t)) \cdot B(g(t)) + \nabla_{g(t)}^2 \cdot B(g(t)).
    \end{cases}
\end{equation}
We will consider its initial value problem $g(0) = g_0$, assuming scalar curvature $S_{g_0} = s_0$. 

If we have a solution $(g(t), p(t))$ of (\ref{eq mod cbf DeTurck flow}), using the vector field $W_g$ 
we may define an one-parameter family of diffeomorphism $\varphi_t: M \to M, \, t \in [0,T]$, 
by solving ordinary differential equations for each $x \in M$
\begin{equation} \label{eq one parame diffeom}
\partial_t \varphi_t(x) = -W_{g(t)} (\varphi_t(x)) \quad \text{with} \quad \varphi_0(x) =x. 
\end{equation}
\begin{lemma} \label{lem the equi before after DeTurck}
(i) Let $(g(t), p(t))$ be a complete solution of (\ref{eq mod cbf DeTurck flow}) on manifold $M^n$ 
with initial metric $g_0$ whose  scalar curvature  $S_{g_0} = s_0$. Let $\varphi_t$ be the solution 
of (\ref{eq one parame diffeom}). We define $\hat{g}(t)= \varphi_t^*g(t) $ and $\hat{p}(x,t) =
p (\varphi_t(x), t)$. Then $(\hat{g}(t), \hat{p}(t))$ is a solution of the modified CBF (\ref{eq cbf 
ellp Baqua and Helli}) with initial condition $g(0) = g_0$.

\noindent  (ii)  Let $(\hat{g}(t), \hat{p}(t))$ be a solution of  
(\ref{eq cbf ellp Baqua and Helli}) with initial metric $g_0$ whose  scalar curvature  $S_{g_0} = s_0$.
As in \cite[p.117]{CLN6} we conclude that equation (\ref{eq one parame diffeom}) is equivalent to
fourth order parabolic equation of maps $\phi_t: M \to M $
\begin{equation} \label{detruck harmonic map cbf}
\partial_t \phi_t =-\Delta_{\hat{g}(t)} \Delta_{\hat{g}(t), \tilde{g}} \phi_t + \frac{n-2}{2(n-1)}
\nabla_{\hat{g}(t)} S_{\hat{g}(t)},
\end{equation}
where $\Delta_{\hat{g}(t), \tilde{g}}$ is the map Laplacian. 
If we have a solution of (\ref{detruck harmonic map cbf}) with $\phi_0=\operatorname{Id}$, 
then one can verify that $((\phi_t^{-1})^* \hat{g}(t),\hat{p}(\phi_t^{-1}(x),t) )$ 
is a solution of  (\ref{eq mod cbf DeTurck flow}) with initial metric $g_0$.
\end{lemma}

\begin{proof} 
(i) A direct computation gives us:
\begin{align*}
& \partial_t \hat{g}(t)= \varphi_t^* (\partial_t g(t)) + \partial_s|_{s=0} 
\left ( \varphi_{t+s}^*g(t)  \right )  \\
= &  2(n-2)  \varphi_t^* \left ( B(g)+ \frac{1}{2(n-1)(n-2)} (\Delta_g S_g) g 
+p g \right) +  \varphi_t^* ( \mathcal{L}_{W_{g}}g ) - \mathcal{L}_{(
\varphi_t^{-1})_* W_{g}} (\varphi_t^* g )  \\
=& 2(n-2)   \left ( B(\varphi_t^*g)+ \frac{1}{2(n-1)(n-2)} (\Delta_{\varphi_t^*g} 
S_{\varphi_t^*g})  \varphi_t^*g+\hat{p} \varphi_t^*g \right ).
\end{align*}
This verifies the first equation in (\ref{eq cbf ellp Baqua and Helli}). 
The second equation in (\ref{eq cbf ellp Baqua and Helli}) for $ \hat{p}(t)$  follows from
 the second equation in (\ref{eq mod cbf DeTurck flow}) for $p(t)$ directly.

(ii) This follows from calculations which are similar to those in (i).
\end{proof}

%%%%%%%
\vskip .2cm
We need some preparations before we engage into the actual proof of the short-time 
existence of CBF  (\ref{eq cbf eq p ellipt}).
Below the H\"{o}lder exponent $\alpha \in (0,1)$.
 We have the following inequality (see \cite[Lemma 3.5]{LQZ14} along with the definitions of 
  H\"{o}lder spaces and Sobolev spaces for tensors).
  
\begin{lemma} \label{lem h time t lip property cbf}
Assume $k \geq 2$. There is a constant $C$ independent of $T$ such that for 
any $t_1, t_2 \in [0,T]$, we have
\begin{equation*}
\|h(\cdot, t_1)- h( \cdot, t_2) \|_{C^{k-2 +\alpha}(M)} \leq C|t_1 -t_2| \cdot \| h \|_{C^{k+\alpha,1}}
\end{equation*}
for any tensor $h \in C^{k +\alpha,1}(M \times [0,T])$.
\end{lemma}

As the starting point to solve the CBF  (\ref{eq cbf eq p ellipt}) we need
the solvability of linear parabolic  system (see \cite{Ha75}, \cite[p.424]{LQZ14},
 or \cite[Proposition 3.2 and 3.3]{BH11}).

\begin{lemma} \label{lem sola standard parabolic linear system high order}
(i) Let $L$ be a linear elliptic differential operator of order $2m$ with $C^{\alpha,0}$-coefficients 
on a closed manifold $M^n$ acting on tensors (see \cite[p.2195]{BH11} for the local expression
 of $L$). Then for every $f \in C^{\alpha, 0} (M \times [0,T] )$ there is a unique solution
  $ u \in C^{2m+\alpha, 1} (M \times [0,T] )$ of 
\[
\partial_t u =Lu + f, \quad u(0, \cdot) =0.
\]
And there is a constant $C$ such that 
\[
\| u \|_{C^{2m+\alpha,1}} \leq C \|f \|_{C^{\alpha,0}}
\]
for all $f \in C^{\alpha, 0} (M \times [0,T] )$.

(ii) Moreover if the coefficient functions of operator $L$ and $f$ in (i) have better 
regularity of $C^{k+\alpha,0}$, then $ u \in C^{2m+k+\alpha, 1} (M \times [0,T] )$
 and we have Schauder estimate 
\[
\| u \|_{C^{2m+k+\alpha,1}} \leq C \|f \|_{C^{k+\alpha,0}}.
\]
\end{lemma}

Because  (\ref{eq mod cbf DeTurck flow}) is a coupled system of $(g(t), p(t))$, 
we further need to reduce it to an equation of $g(t)$ only and then solve the resulting
 equation for the short-time existence. The following simple lemma is \cite[Lemma 3.4]{LQZ14}.

\begin{lemma} \label{lem solve the ellipt eq of p}
(i) Suppose that ${g}(t), \, t \in [0,T]$, is a family of $C^{1+\alpha, 0}$-metric
 such that the elliptic operator $(n-1) \Delta_{{g}(t)} +s_0$ is invertible for each $t$. 
 Then there is a constant  $C$ such that for each $\gamma \in  C^{\alpha,0}(M \times [0,T])$ 
 equation 
\[
\left ( (n-1) \Delta_{{g}(t)} +s_0 \right ) p(t)  = \gamma(t)
\]
has a unique solution $p \in C^{2+\alpha,0}(M \times [0,T])$, and  
\[
\| p\|_{C^{2+\alpha,0} } \leq C \| \gamma \|_{C^{\alpha,0}}.
\]

(ii) Moreover if ${g}(t)$ is in $C^{k-1+\alpha, 0}$ and $\gamma$ is in $C^{k-2+\alpha, 0}$
 for some $k \geq 2$, then we have the solution $p \in C^{k+\alpha,0}(M \times [0,T])$ and estimate
\[
\| p\|_{C^{k+\alpha,0} } \leq C \| \gamma \|_{C^{k-2+\alpha,0}}.
\]
\end{lemma}

Consider metric $g \in C^{4+\alpha}(M)$ such that the operator $(n-1) \Delta_{g}+s_0$ is  invertible,
we may define an operator $\mathcal{P}( \cdot )$ such that  $\mathcal{P}(g) \in 
  C^{2+\alpha}(M)$ is the solution $p$ of the equation 
\begin{equation} \label{eq def of mathcal P of p}
 \left ( (n-1)\Delta_{g} + s_0 \right ) p =- (n-2) A(g) \cdot B(g )  
   + \nabla_{g}^2 \cdot B(g) \in C^{\alpha}(M).
\end{equation}

\begin{lemma} \label{lem lip property of mathcal P op}
Suppose that $M^n$ is a closed manifold with metrics $g \in C^{4+\alpha, 0}(M \times [0, T])$.
 Suppose that the elliptic operator $(n-1)\Delta_{g(t)} +s_0$ is invertible for each $t \in [0, T]$. 
 Then there are constants  $C >0$ and small $\delta_0 >0$ depending on $g$ such that 
\begin{equation} \label{eq Holder P op bound}
\|\mathcal{P}(g_1) -\mathcal{P}(g_2) \|_{C^{2+\alpha, 0}}  \leq C \| g_1 - g_2 \|_{C^{ 4+ \alpha, 0}}
\end{equation}
for all $g_i \in C^{4+\alpha, 0}(M \times [0, T])$ which satisfies $\|g_i -  g\|_{C^{4+\alpha,0}}
 \leq \delta_0$ for $i=1, 2$. We also have in Sobolev norm
\begin{equation} \label{eq sobolev P op bound}
\|\mathcal{P}(g_1) -\mathcal{P}(g_2) \|_{W^{2, 2}}  \leq C \| g_1 - g_2 \|_{W^{ 4, 2}}.
\end{equation}
\end{lemma}

\begin{proof} 
Let operator $T_1(g) \doteqdot - (n-2) A(g) \cdot B(g) + \nabla_{g}^2 \cdot B(g)$. 
From equation (\ref{eq def of mathcal P of p})  we have
\[
( (n-1) \Delta_{g_1} +s_0)( \mathcal{P} (g_1) -\mathcal{P}(g_2) ) =  (n-1) (\Delta_{g_2} - 
\Delta_{g_1} ) \mathcal{P}(g_2)+ T_1(g_1)-T_1(g_2).
\]
Inequality (\ref{eq Holder P op bound}) follows from Lemma \ref{lem solve the ellipt eq of p}(i).
 Using $L^2$-norm we get (\ref{eq sobolev P op bound}).
\end{proof}

Let $\mathcal{F}(g) \doteqdot 2(n-2) \left (B(g) +  \frac{1}{2(n-1)(n-2)} (\Delta S_{g}) g
 \right)  + \mathcal{L}_{W_{g}} g.$ We will consider the following decoupled DeTurck
  modified CBF of $g(t)$ induced from  (\ref{eq mod cbf DeTurck flow})
\begin{equation} \label{eq decoupled DeTur CBF}
\mathcal{M}(g(t)) \doteqdot \partial_t g(t)  - \mathcal{F}(g(t)) -2(n-2)\mathcal{P}(g(t))g(t) =0
\end{equation}
with the initial condition $g(0) =g_0$, assuming scalar curvature $S_{g_0}=s_0$.
Clearly this flow is equivalent to flow (\ref{eq mod cbf DeTurck flow}).

\vskip .1cm
Next we turn to prove the short-time existence of solutions of the
 linearized equation of  (\ref{eq decoupled DeTur CBF}) by contraction mapping theorem.

%%%%%%%%%%%%%%%%%%%%%%%%%%%
\subsection{Linearization of the decoupled flow (\ref{eq decoupled DeTur CBF})} 
\label{subsect lineard DeTurck mod-cbf}

We now compute the linearization of the decoupled DeTurck modified CBF
 (\ref{eq decoupled DeTur CBF}) by computing the linearization of operators $\mathcal{F}(g)$ 
 and $\mathcal{P}(g)$ separately. 
Let $g_s= g+sh, \, s \in (-\epsilon, \epsilon)$, for some $(0,2)$-tensor $h$.
Recall that Bach tensor has the schematic expression (\ref{eq Bach definition}),
from the linearization formula (\cite[(2.47)]{CLN6} used in the DeTurck trick for Ricci flow, 
we get (we are sketchy on the details of some tedious but routine calculations here)
\begin{equation}\label{eq linearization B1 answer}
\delta_h \mathcal{F}(g)  \doteqdot  \left . \frac{d}{ds}  \right |_{s=0}  \mathcal{F}(g_s) =
 - \Delta_g \Delta_{L,g} h + \sum_{a=0}^3 M_a(g) * \nabla_g^{a} h,
\end{equation}
where the Lichnerowicz operator is defined by
\begin{equation*}
\Delta_L h_{ij} = \Delta_{L,g} h_{ij}  = \Delta_g h_{ij} +2R_{iklj}h_{kl} -R_{ik}h_{kj}-R_{jk}h_{ki},
\end{equation*}
$M_a(g)$ depends on $g$ up to the fourth derivatives, and $*$ is some tensor contraction 
operation by metric $g$ and its inverse.

We define  $\mathcal{P}^{\prime}_g(h)= \delta_h
 \mathcal{P}(g) \doteqdot  \left . \frac{d}{ds}  \right |_{s=0} \mathcal{P}(g_s) $. 
 To find $\mathcal{P}^{\prime}_g(h)$ we compute the linearization of the both sides 
 in equation (\ref{eq def of mathcal P of p}) with $g =g_s$. 
 Let operator $\mathcal{L} \doteqdot (n-1) \Delta_{g}+s_0 $, then
\begin{equation*}
\delta_h \left( \mathcal{L} \mathcal{P}(g) \right ) = - (n-1) \nabla_i \nabla_j  \mathcal{P}(g) 
h_{ij}  - \frac{n-1}{2}\nabla_j  \mathcal{P}(g) \left ( 2 \nabla_i h_{ij} -\nabla_j H \right ) 
+ \mathcal{L} \mathcal{P}^{\prime}_g(h)
\end{equation*}
where $H =g_{ij}h_{ij}$ (\cite[p.547]{CLN6}).
From the right-hand-side of (\ref{eq def of mathcal P of p}), we get
\begin{equation*}
\delta_h \left ( -(n-2) A(g) \cdot B(g) + \nabla_g^2 \cdot  B (g)  \right ) = 
\sum_{a=0}^{4} P_a(g) *  \nabla^{a} h ,
\end{equation*}
where $P_a(g)$'s are tensors depending on  $g$ up to the fourth derivatives. 
Hence we have
\begin{equation}\label{eq linearization mathcal P answer}
\mathcal{P}^{\prime}_g(h) =  \mathcal{L} ^{-1}  \left ( (n-1)  \nabla_i \nabla_j
  \mathcal{P}(g) h_{ij}+ \frac{n-1}{2} \nabla_j  \mathcal{P}(g) \left (2 \nabla_i h_{ij}
   -\nabla_j H \right )  + \sum_{a=0} ^{4} P_a(g) * \nabla^{a} h \right ),
\end{equation}
where we have assumed that operator $\mathcal{L} $ is invertible.
Hence we have
\begin{equation}\label{eq DM oper def akl}
\delta_{h(t)} \mathcal{M}(g(t)) = \partial_t h + \Delta_g \Delta_{L,g} h - \sum_{a=0}^3 
M_a(g) * \nabla_g^{a} h -2(n-2) \mathcal{P}^{\prime}_g(h) g -2(n-2)  \mathcal{P}(g) h.
\end{equation}

The remaining proof for the short-time existence of the decoupled DeTurck modified CBF
(\ref{eq decoupled DeTur CBF})  is very close to the proof for that of CRF
 in \cite[\S 3.3.2-3.3.3]{LQZ14}. First we need to solve the linearized flow
\begin{equation}\label{eq linearized deturck pluged}
   \begin{cases}
        & \partial_t h + \Delta_g \Delta_L h - \sum_{a=0}^3 M_a(g) * \nabla_g^{a} h -2(n-2)
         \mathcal{P}^{\prime}_g(h) g -2(n-2)  \mathcal{P}(g) h  = \gamma\\
        & h (\cdot, 0) = 0
    \end{cases}
\end{equation}
for each $\gamma \in C^{\alpha,0}(M \times [0,T])$. Using Lemma \ref{lem sola standard
 parabolic linear system high order}(i) we have
 
\begin{proposition} \label{lem 4}
Suppose that $g(t)$, $t\in [0, T]$,  is a family of  $C^{4+\alpha, 0}$-metrics such that the
 elliptic operator $(n-1)\Delta_{g(t)}+s_0$ is invertible for all $t$. Then for each $\gamma 
 \in C^{\alpha, 0}(M \times [0, T])$ the initial value problem (\ref{eq linearized deturck pluged})
  has a unique solution $h \in C^{4+ \alpha, 1}(M \times [0, T])$. 
  Moreover there is a constant $C$ such that
\begin{equation}\label{est of sol in haf cbf}
\|h\|_{C^{4+\alpha,1}}\leq  C \| \gamma \|_{ C^{\alpha, 0}}
\end{equation}
for all $\gamma \in C^{\alpha, 0}(M \times [0, T])$.
\end{proposition}
\begin{proof} 
To use contraction mapping theorem to prove the existence  we consider the Banach space
\[
E_1([0,T^*]) =\left  \{ \tilde{h} \in C^{4+\alpha, 0}(M \times [0, T^*]): 
\tilde{h}( \cdot, 0) = 0 \right \},
\]
where $T^* \in (0,T]$ is a small constant to be chosen below. By Lemma 
\ref{lem sola standard parabolic linear system high order}(i) for a given $\tilde{h}
 \in E_1([0, T^*])$ we can solve the system of linear parabolic equations
\begin{equation} \label{linearized Deturck cbf akls}
\begin{cases}
& \partial_t h + \Delta_g \Delta_L h - \sum_{a=0}^3 M_a(g) * \nabla_g^{a} h-2(n-2)  
\mathcal{P}(g) h  =\tilde{\gamma} \\
& h(\cdot, 0) =0,
\end{cases}
\end{equation}
where $\tilde\gamma \doteqdot \gamma + 2(n-2){\mathcal P}^{\prime}_g(\tilde h) g
 \in C^{\alpha,0}$, hence we may define a map
\begin{equation}\label{eq fixed point map Psi}
\Psi: E_1([0,T^*]) \rightarrow E_1([0,T^*]), \quad \Psi (\tilde{h}) =h \in C^{4+\alpha,1}.
\end{equation}
 
Let $\tilde{h}_i \in E_1([0,T^*]), \, i=1, 2$. Note that if we set $v \doteqdot \Psi(\tilde{h}_1)
 - \Psi(\tilde{h}_2),$ then $v$ satisfies
\begin{equation*}
\begin{cases}
&\partial_t v +  \Delta_g \Delta_L v -\sum_{a=0}^3 M_a(g)* \nabla_g^a v -2(n-2)
 \mathcal{P}(g) v = 2(n-2) ({\mathcal P}_g^{\prime}(\tilde h_1)
 -{\mathcal P}_g^{\prime}(\tilde h_2)) g, \\
& v(\cdot,0) =0.
\end{cases}
\end{equation*}
From (\ref{eq linearization mathcal P answer}) and Lemma \ref{lem solve the ellipt eq of p}(i)
 we have for $g \in C^{4+\alpha, 0}$
\begin{equation} \label{eq cal P deriv Lip prop}
\| \mathcal {P}_g^{\prime}(\tilde{h}_1) - \mathcal{P}_g^{\prime}(\tilde{h}_2) \|_{C^{2
+\alpha,0}} \leq C \|\tilde h_1 - \tilde h_2\|_{C^{4+\alpha,0}},
\end{equation}
then it follows from Lemma \ref{lem sola standard parabolic linear system high order}(ii) that
\begin{equation} \label{eq est v tem use 1}
\|v\|_{C^{6+\alpha,1}} \leq C \|\tilde h_1 - \tilde h_2\|_{C^{4+\alpha,0}}.
\end{equation}
Hence by Lemma \ref{lem h time t lip property cbf} we have
\[
\|v(\cdot, t_1) - v(\cdot, t_2) \|_{C^{4+\alpha}} \leq C\cdot |t_1-t_2|\cdot 
\|\tilde h_1 - \tilde h_2\|_{C^{4+\alpha,0}}.
\]
In particular, using $t_1 =0$ and $v(\cdot, 0)=0$ we get
\begin{equation} \label{eq used for uniqueness linear cal P}
\|\Psi(\tilde h_1) - \Psi(\tilde h_2)\|_{C^{4+\alpha,0}} \leq CT^*
\|\tilde h_1 - \tilde h_2\|_{C^{4+\alpha,0}}
\end{equation}
for all $\tilde{h}_i \in E_1([0,T^*]), \, i=1, 2$.

To apply contraction mapping theorem to $\Psi$  we observe that
\begin{equation} \label{eq thm 5.1.21 est source local}
\|\Psi(\tilde h)\|_{C^{4+\alpha,0}} \leq \|\Psi(0)\|_{C^{4+\alpha,0}} + CT^*\|\tilde h\|_{C^{4+\alpha,0}}
\end{equation}
by (\ref{eq used for uniqueness linear cal P}) and that for some constant $C_0$
\[
\|\Psi(0)\|_{C^{4+\alpha, 1}} \leq C_0 \|\gamma\|_{C^{\alpha,0}}
\]
by Lemma \ref{lem sola standard parabolic linear system high order}(i). 
Let $R=2C_0\|\gamma\|_{C^{\alpha, 0}}$, then when $T^*$ is chosen
 so that $CT^* \leq \frac{1}{2}$, the map
\[
\Psi: \bar{B}_R = \{\tilde{h} \in E_1([0, T^*]): \| \tilde{h} \|_{C^{4+\alpha, 0}}
\leq R\} \, \hookrightarrow  \, \bar{B}_R
\]
is a contractive mapping due to (\ref{eq used for uniqueness linear cal P}).  
We get a fixed point of $\Psi$ on $\bar{B}_R$
 which gives the existence of the solution of equations (\ref{eq linearized deturck pluged}) 
 on time interval $[0,T^*]$.

\vskip .1cm
To see the uniqueness of solutions to (\ref{eq linearized deturck pluged}), 
suppose that $h_1$ and $h_2$ are two solutions, it follows from
 (\ref{eq used for uniqueness linear cal P}) using $\tilde{h}_i =h_i = \Psi(\tilde h_i)$ 
 and  $CT^*  \leq \frac{1}{2}$ that $h_1 -h_2 =0$. 
 
 Because (\ref{eq linearized deturck pluged}) 
 is linear, there will be no short-time blowup, one may extend its solution from $[0, T^*]$ 
 to $[0, T]$ by steps over time intervals of length $T^*$. 
 Note that when we extend the solution to $[T^*, 2T^*]$, we need to make some simple
  adjustment to the equations so that the initial condition at $T^*$ for the new equations is 0. 
  
  The estimate (\ref{est of sol in haf cbf}) follows from the estimate
   (\ref{eq thm 5.1.21 est source local}).
\end{proof}

In summary we have established that the linear operator defined by (\ref{eq DM oper def akl})
\begin{equation} \label{eq 6 alpha delta M invertible 6 to 4}
\delta_{\bullet}\mathcal{M}(g): \left \{ h\in C^{4+\alpha,1}(M \times [0, T]), \, 
h(\cdot, 0) =0 \right \}\to C^{\alpha, 0}( M \times [0, T])
\end{equation}
is an isomorphism, provided that $g=g(t)$ satisfies the assumptions in Proposition \ref{lem 4}.

%%%%%%%%%%%%%%%%%%%%%%%%%%%
\subsection{Short-time existence of decoupled DeTurck modified CBF 
(\ref{eq decoupled DeTur CBF}) and proof of the existence part of
Theorem \ref{thm main short time exist cbf}} 
 \label{subsect short time decouple DeTurck cbf}
Now we apply the implicit function theorem \cite[Lemma 3.7]{LQZ14} to the nonlinear map
defined by (\ref{eq decoupled DeTur CBF})
\[
\mathcal {M}: \left \{g \in C^{4+\alpha,1}( M \times  [0, T]), \, 
g(0) = g_0\right \} \to C^{\alpha,0}( M \times [0, T]).
\]
Here $g_0$ is the metric given in Theorem \ref{thm main short time exist cbf} and
we choose the metric $\tilde{g}=g_0$ which is used in (\ref{eq vector field w def}).
We begin with showing that operator $\mathcal{M}$ is continuously differentiable.

\begin{lemma}\label{lem continuity delta M} 
Let $M^n$ be a closed manifold with 
 metrics $g(t)\in C^{4+\alpha, 1}(M \times [0, T])$. Suppose that the elliptic operator
$(n-1)\Delta_{g(t)} +s_0$ is invertible for each $t$. 
Then there is a constant $\delta_0>0$ such that we have the following estimate of the norm of the 
difference of the linear operators
\[
\|\delta_{\bullet} \mathcal {M}(g_1) - \delta_{\bullet} \mathcal {M}(g_2)\|_{L(C^{4+\alpha,1}, 
C^{\alpha, 0})} \leq C\|g_1 - g_2\|_{C^{4+\alpha, 1}}
\]
for  all $g_i \in C^{4+\alpha, 1}(M \times [0, T])$ which satisfies  
$\|g_i -  g\|_{C^{4+\alpha,1}} \leq \delta_0, \, i=1, 2$.
\end{lemma}

\begin{proof} For any $ h \in C^{4+ \alpha,1} (M \times [0,T])$ with $h(\cdot, 0) =0$ 
we have 
\begin{align*}
& \delta_h \mathcal{M}(g_1) - \delta_h \mathcal{M}(g_2) \\
 = & \left ( \Delta_{g_1}\Delta_{L,g_1} 
- \Delta_{g_2}\Delta_{L,g_2} \right ) h- \sum_{a=0}^3 \left (M_a(g_1) *_{g_1}  \nabla_{g_1}^a h
- M_a(g_2) *_{g_2}  \nabla_{g_2}^a h \right ) \\
& -2(n-2) \left (\mathcal{P}^{\prime} _{g_1}(h) g_1 -\mathcal{P}^{\prime}_{g_2}(h) g_2 \right )
- 2(n-2) \left ( \mathcal{P}(g_1) -\mathcal{P}(g_2) \right )h.
\end{align*}
We will estimate the operator norm of each term in the above expression,
and omit most of the tedious but routine calculations below.

First we have
\begin{align} 
& \| \left ( \Delta_{g_1}\Delta_{L,g_1} - \Delta_{g_2}\Delta_{L,g_2} \right ) h \|_{
        C^{\alpha, 0}}  \notag \\
 \leq & \| (\Delta_{g_1} - \Delta_{g_2} )\Delta_{L,g_1} h \|_{C^{\alpha, 0}} +
 \| \Delta_{g_2}(\Delta_{L,g_1} - \Delta_{L,g_2} ) h \|_{C^{\alpha, 0}}   \notag \\
\leq & C \| g_1 - g_2 \|_{C^{4+\alpha,0}} \cdot   \| h \|_{C^{4+\alpha,0}}.
\label{eq temp cal Lichnero Lapla differe}
\end{align}
For each $a =0,1,2,3$ we have
\begin{align}
& \| M_a(g_1) *_{g_1}  \nabla_{g_1}^a h - M_a(g_2) *_{g_2}  \nabla_{g_2}^a h
 \|_{C^{\alpha, 0}} \notag  \\
 \leq & \| (M_a(g_1) *_{g_1} -M_a(g_2)) *_{g_2}) \nabla_{g_1}^a h  \|_{C^{\alpha, 0}}
  + \| M_a(g_2)*_{g_2}( \nabla_{g_1}^ah  -\nabla_{g_2}^ah ) \|_{C^{ \alpha, 0}} \notag \\
\leq &  C \| g_1 - g_2 \|_{C^{4+\alpha,0}} \cdot   \| h \|_{C^{3+\alpha,0}} 
+C \| g_1 - g_2 \|_{C^{3+\alpha,0}} \cdot   \| h \|_{C^{3+\alpha,0}} .
\label{eq temp cal nabla a h terms}
\end{align}
From a proof similar to that on the bottom of \cite[p.426]{LQZ14} we have
\[
 \|  \mathcal{P}^{\prime} _{g_1}(h)  -\mathcal{P}^{\prime}_{g_2}(h)  \|
_{C^{\alpha, 0}}  \leq C \| g_1 - g_2 \|_{C^{2+ \alpha, 0}} \cdot \| 
h \|_{C^{2+\alpha,0}},
\]
hence by using (\ref{eq cal P deriv Lip prop}) with $\tilde{h}_2 =0$ we get
\begin{align}
        &\| \mathcal{P}^{\prime} _{g_1}(h) g_1 -\mathcal{P}^{\prime}_{g_2}(h)
         g_2\|_{C^{ \alpha, 0}}  \notag \\
        \leq &  C \| g_1 - g_2 \|_{C^{2+ \alpha, 0}} \cdot \| h \|_{C^{2+\alpha,0}}
        + C \| g_1 - g_2 \|_{C^{\alpha, 0}} \cdot \| h \|_{C^{2+\alpha,0}}.
        \label{eq nonlocal term cal M lip 1}
\end{align}
Finally using Lemma \ref{lem lip property of mathcal P op} we have 
\begin{align}
\| \left ( \mathcal{P}(g_1) -\mathcal{P}(g_2) \right )h  \|_{C^{2+\alpha, 0}} \leq  
C \| g_1 - g_2 \|_{C^{4+ \alpha, 0}} \cdot \| h \|_{C^{2+\alpha,0}}. 
\label{eq last term to est}  
\end{align}

The lemma now follows from combining together the inequalities 
(\ref{eq temp cal Lichnero Lapla differe}), (\ref{eq temp cal nabla a h terms}), 
(\ref{eq nonlocal term cal M lip 1}), and (\ref{eq last term to est}).
\end{proof}

\vskip .1cm
Next we prove the short-time existence of flow
(\ref{eq decoupled DeTur CBF}).

\begin{proposition} \label{prop short time sol decoupled DeTurck CBF}
Let $M^n$ be a closed manifold. Suppose that $g_0$ is a $C^{8+\alpha}$-Riemannian 
 metric on $M$ with constant scalar curvature $s_0$ and that the elliptic operator
  $(n-1) \Delta_{g_0} + s_0$ is invertible. Then there exists a unique 
  $C^{4+\alpha, 1}$-solution $g(t), \, t \in [0,T]$, of the decoupled DeTurck modified
   CBF (\ref{eq decoupled DeTur CBF}) for some $T>0$.
\end{proposition}

\begin{proof}   
Let 
\begin{equation} \label{eq M opera in DeTurck decoup cbf q}
\mathcal{M}: \{ g \in C^{4+\alpha,1} ( M \times [0,T]), \,  g(0) =g_0  \}
\rightarrow  C^{ \alpha, 0} ( M \times [0,T])
\end{equation}
be a map defined in (\ref{eq decoupled DeTur CBF}). 
We define metrics $\bar{g} \in C^{4+\alpha,1} ( M \times [0,T])$ by
\begin{equation}\label{eq metric bar g as x_0}
\bar{g}(t) = g_0 + t  (\mathcal{F}(g_0) + 2(n-2)\mathcal{P}(g_0)g_0) .
\end{equation}
We will apply inverse function theorem \cite[Lemma 3.7]{LQZ14} to map $\mathcal{M}$ 
around $\bar{g}$  to prove the existence of a solution of equation $\mathcal{M}(g)=0$. 
Note that linear operator $(n-1)\Delta_{\bar{g}(t)} + s_0$ is a small
perturbation of  $(n-1)\Delta_{g_0} + s_0$ when $t$ is small.
 Since operator $(n-1)\Delta_{g_0} + s_0$ is invertible,
  we conclude that $(n-1)\Delta_{\bar{g}(t)} + s_0$ is invertible for each $t \in [0,T]$ 
  when $T$ is small enough. In general metric $\bar{g}(t)$ does not have constant
   scalar curvature for $t>0$. 

It follows from a simple calculation that  
\begin{equation}\label{eq approx solu error H(x_0)}
\mathcal{M} (\bar{g}(t))  =  \mathcal{F}(g_0) + 2(n-2)\mathcal{P}(g_0)g_0 - \left (\mathcal{F}
(\bar{g}(t))  + 2(n-2)\mathcal{P}(\bar{g}(t) \right )\bar{g}(t)).
\end{equation}
Using Lemma \ref{lem lip property of mathcal P op} we get
\begin{equation} \label{eq erro of approxmate solu cal M}
\|\mathcal {M}(\bar{g})\|_{C^{ \alpha, 0}} \leq  CT
\end{equation}
where constant $C$ depends on $\| g_0 \|_{C^{8+\alpha}}$. Hence $\bar{g}(t)$ is an 
approximate solution of  equation $\mathcal{M}(g)=0$ when $T$ is small.

By (\ref{eq 6 alpha delta M invertible 6 to 4}) we have that for sufficient small $T$
\[
\|(\delta_{\bullet} \mathcal {M} (\bar{g}))^{-1}\|_{L(C^{\alpha,0}, C^{4+\alpha, 1})} 
\leq \bar{C}
\]
for some constant $\bar{C}$.
 Let 
 \[
 \bar{B}(\bar{g}, \delta_0) = \left \{g\in C^{ 4+\alpha, 1}(M \times [0,T]): 
 \, g(0) =g_0 \text{ and }  \|g -\bar{g}\|_{C^{4+\alpha, 1}} \leq \delta_0 \right \}.
 \] 
  By the perturbation theory
  of bounded linear operators and Lemma \ref{lem continuity delta M} there is a constant $C_0$ 
  and a small number $\delta_0>0$  such that the operator norm
\[
\|(\delta_{\bullet} \mathcal {M} (g))^{-1}\|_{L(C^{\alpha,0}, C^{4+\alpha, 1})} \leq C_0
\]
for all $g \in \bar{B}(\bar{g}, \delta_0)$. By Lemma \ref{lem continuity delta M} we can choose 
$\delta_0$ even smaller if necessarily such that  for the constant $C_0$ above we have
\[
\| \delta_{\bullet} \mathcal {M} (g_1) -\delta_{\bullet} \mathcal {M}(g_2)\|_{L( C^{4+\alpha, 1},
 C^{\alpha,0})}   \leq \frac{1}{2C_0}
\]
for all $g_1, g_2 \in \bar{B}(\bar{g}, \delta_0)$. From (\ref{eq erro of approxmate solu cal M})
 we may choose an even smaller $T$ if necessary to get
\[
\|\mathcal {M}(\bar g)\|_{C^{\alpha, 0}} \leq \frac {\delta_0} {2C_0}.
\]
Now the existence in Proposition \ref{prop short time sol decoupled DeTurck CBF} follows from
 the inverse function theorem.
 
 \vskip .1cm
We leave the proof of  the uniqueness to subsection \S \ref{subsec uniquness cbf feb15}. 
\end{proof}

\vskip .2cm
Finally we can give a quick

\noindent {\it Proof of the existence part of Theorem \ref{thm main short time exist cbf}}.
First we assume that the initial metric  $g_0$ is a $C^{8+\alpha}$, by
Proposition \ref{prop short time sol decoupled DeTurck CBF} we have a solution 
$g\in C^{ 4+\alpha, 1}(M \times [0,T])$ of (\ref{eq decoupled DeTur CBF})
and consequently a function  $p \in C^{2+\alpha,0}$. 
Near the end of \S \ref{subsect DeTurck cbf} we conclude the equivalence between
 (\ref{eq decoupled DeTur CBF}) and (\ref{eq mod cbf DeTurck flow}), 
 hence we get a solution $(g(t),p(t))$ of (\ref{eq mod cbf DeTurck flow}). 
  The better regularity  $ g \in C^{4+\alpha, 1+ \frac{1}{4}\alpha}$ 
   follows from the standard parabolic theory. 
  By Lemma \ref{lem the equi before after DeTurck}(i) and \ref{lem CBF equiv forms}
  we get the required solution of  CBF  (\ref{eq cbf eq p ellipt}). 
  
  \vskip .1cm
  When the initial metric  $g_0$ is in $C^{4+\alpha}$, from the solution of Yamabe probelm
  we may choose a family of  $C^{8+\alpha}$-metrics  $g_{0i}$ of constant scalar curvature 
  $s_0 \neq 0$ which converges to  $g_0$ in $C^{4+\alpha}$-norm.  
Hence we have a family of solutions $g_i(t)$ in $C^{ 4+\alpha, 1+ \frac{1}{4}\alpha}(M \times [0,T])$
 of flow  (\ref{eq decoupled DeTur CBF}).
 
By Lemma \ref{lem continuity delta M}  the operator $\mathcal{M}(g)$ is continuously differentiable 
in $g \in C^{4+\alpha,1} ( M \times [0,T])$. Hence by the continue-dependence on parameters in 
the inverse function theorem and the proof of the existence of $g_i(t)$ in Proposition
\ref{prop short time sol decoupled DeTurck CBF}
we conclude that the sequence of fixed points  $\{ g_i(t) \}$ converges in  
$C^{4+\alpha,1}(M \times [0,T])$-norm to
 a solution $g_{\infty}(t) \in C^{ 4+\alpha, 1}(M \times [0,T])$
 of (\ref{eq decoupled DeTur CBF})  with initial metric $g_{\infty}(0) =g_0$.
 This  $g_{\infty}(t)$ gives rise to the required solution of CBF (\ref{eq cbf eq p ellipt}) 
 with initial metric $g_0$. 
  \hfill $\square$

%%%%%%%%%%%%%
\subsection{Proof of the uniqueness in Theorem \ref{thm main short time exist cbf}}
 \label{subsec uniquness cbf feb15}
 
 First we prove the uniqueness in Proposition \ref{prop short time sol decoupled DeTurck CBF}
 by energy method. Since such proof is well-understood in parabolic equations and the calculation 
 is tedious here, we will be sketch on details.

Suppose $g_1(t)$ and $g_2(t)$ are two solutions with initial data $g_1(0) =g_2(0) =g_0$. 
Let $g_{21}(t) \doteqdot g_2(t)-g_1(t)$. Using $\mathcal{F}(g_2) -\mathcal{F}(g_1) =\int_0^1 \frac{d}{ds}
\mathcal{F}(g_1+s(g_2 -g_1)) ds$  and (\ref{eq linearization B1 answer}) we have
\begin{align*}
\partial_t g_{21}
=&  - \Delta^2_{g_1}  g_{21}+  \sum_{a=0}^3 \check{A}_a *\nabla^a g_{21}
 -2(n-2) \mathcal{P}(g_1) g_{21}  -2(n-2)  ( \mathcal{P}(g_2) - \mathcal{P}(g_1)) g_2 .
\end{align*}

Let $U(t) \doteqdot \int_M \sum_{b=0}^2|\nabla_{g_1(t)}^bg_{21}(t)|^2_{g_1(t)} d \mu_{g_1(t)}$ 
be the energy and let $[A, B] \doteqdot AB -BA$ be the commutator of two operators.
We compute schematically the derivative of $U(t)$ using integration by parts,
\begin{align*}
 \frac{dU(t)}{dt}  \leq & 2 \int_M \sum_{b=0}^2  \langle \left ( \nabla_{g_1(t)}^b \partial_t+
  [\partial_t, \nabla_{g_1(t)}^b] \right) g_{21},  \nabla_{g_1(t)}^bg_{21}(t) \rangle _{g_1(t)} 
 d \mu_{g_1(t)}  + C U(t)  \\
 \leq &   - 2 \int_M \sum_{b=0}^2  \langle  \nabla_{g_1(t)}^b\Delta^2_{g_1} 
 g_{21}, \nabla_{g_1(t)}^b g_{21}(t) \rangle_{g_1(t)} d \mu_{g_1(t)}   \\
& + 2 \int_M \sum_{a=0}^5 \sum_{b=0}^2  \langle \tilde{A}_{a,b} * \nabla
_{g_1(t)}^a g_{21}, 
\nabla_{g_1(t)}^bg_{21}(t)  \rangle _{g_1(t)} d \mu_{g_1(t)}   \\
& - 4(n-2)  \int_M \sum_{b=0}^2 \langle \nabla_{g_1(t)}^b ( ( \mathcal{P}(g_2)
-   \mathcal{P}(g_1)) g_2 ), \nabla_{g_1(t)}^b g_{21} \rangle_{g_1(t)}   d \mu_{g_1(t)} 
 + C U(t) \\ 
\leq &    - \int_M \sum_{b=0}^2 | \Delta_{g_1}  \nabla_{g_1(t)}^b g_{21}|^2_{g_1(t)} 
d \mu_{g_1(t)}+  C U(t).
\end{align*}
Here we omit all the details to obtain the last inequality.
Since  $U(0) =0$  by $g_{21} (0) =0$,
we have $U(t) =0$ from Gronwall inequality and the uniqueness is proved.

 \vskip .2cm
 As a consequence we give a 
 
\noindent {\it Proof of the uniqueness in Theorem \ref{thm main short time exist cbf}}.
The proof is standard as the proof of the uniqueness of Ricci flow on closed manifolds
 (see, for example, \cite[p.117-118]{CLN6}). The basic idea is that given two solutions 
 $(g_i(t),p_i(t)), \, i=1, 2$, of CBF (\ref{eq cbf eq p ellipt}),  from Lemma
  \ref{lem the equi before after DeTurck}(ii) we have two diffeomorphisms  $\varphi_i(t)$ 
  which are solutions of the parabolic equations (\ref{detruck harmonic map cbf}) 
   corresponding to $g_i(t)$. 
   Then the pushing-forward metrics $(\varphi_i(t))_* g_i(t)$ are solutions of DeTurck CBF
    (\ref{eq decoupled DeTur CBF}) satisfying the same initial condition, hence by the
     uniqueness in Proposition \ref{prop short time sol decoupled DeTurck CBF}
      we have $( \varphi_1(t))_* g_1(t) =( \varphi_2(t))_* g_2(t) =g_*(t)$. 
      Note that $\varphi_i(t)$ are the solutions of ODE (\ref{eq one parame diffeom}) 
      for metric $g_*(t)$ with initial condition $\varphi_1(0)=\varphi_2(0) =\operatorname{Id}_M$. 
      Hence $\varphi_1(t)=\varphi_2(t)$ and  $g_1(t) = \varphi_1(t)^* 
      g_*(t) =\varphi_2(t)^* g_*(t) =g_2(t)$.
\hfill $\square$

\begin{remark} One may give a direct proof of the uniqueness in 
Theorem \ref{thm main short time exist cbf} using Kotschwar's energy 
techniques without relying on DeTurck CBF  (\cite{Ko14}).
More precisely, for two solution $g_1(t)$ and $g_2(t)$ of CBF  (\ref{eq cbf eq p ellipt})  
with common initial  data $g_0$ one may use energy defined by
\begin{align*}
 \mathcal{E}(t) \doteqdot &   \int_M |g_1(t)-g_2(t)|^2_{g_1(t)} d \mu_{g_1(t)} 
+ \int_M |\nabla_{g_1(t)}-  \nabla_{g_2(t)}|^2_{g_1(t)} d \mu_{g_1(t)} \\
& + \sum_{a=0}^1 \int_M |\nabla^a_{g_1(t)} \operatorname{Rm}_{g_1(t)}- 
 \nabla^a_{g_2(t)} \operatorname{Rm}_{g_2(t)}|^2_{g_1(t)} d \mu_{g_1(t)}.
\end{align*}
\end{remark}

%%%%%%%%%%%%%%
\subsection{Regularity}
We prove the following theorem.
\begin{theorem} \label{thm main cbf better regu}
Suppose that $g_0$ is smooth. Then the solution $(g(t), p(t))$  of CBF  (\ref{eq cbf eq p ellipt})
is smooth in space and time  variables for a short time.
\end{theorem}

\begin{proof} 
Let metric $\tilde{g}$ in (\ref{eq vector field w def}) be $g_0$. Using local coordinates
 $(x^i)$ we may rewrite the DeTurck modified CBF (\ref{eq mod cbf DeTurck flow})
  in a schematic way  as
\begin{equation} \label{eq cbf in local for regula}
    \begin{cases}
       \partial_t g_{ij} + (g^{kl}\partial_k \partial_l)  (g^{pq}\partial_p \partial_q) g_{ij}
      +  \hat{A}_{ij}(g_0,g)  -2(n-2) pg_{ij} = 0,   \\
        ((n-1)\Delta_{g(t)} + s_0) p = -(n-2)A(g(t)) \cdot B(g(t))+  \nabla_{g(t)}^2 B(g(t)),
    \end{cases}
\end{equation}
where $ \hat{A}_{ij}(g_0,g)$ depends on $\{g_{pq}\}$ up to their third derivatives. 
We want to prove $\partial^a g_{ij}$ is in $C^{4+ \alpha, 1}(M \times [0,T])$
 for each $a \in \mathbb{N}$ by a bootstrap argument. 

Below we only consider the base case $a= 1$, as an example. By the existence 
part of Theorem \ref{thm main short time exist cbf} solution $g \in C^{4+\alpha,1}$
 and $p \in C^{2+\alpha,0}$, hence $ \hat{A}_{ij}(g_0,g)  -2(n-2) pg_{ij} \in C^{1+\alpha,0}$.
  It follows from Lemma \ref{lem sola standard parabolic linear system high order}(ii) and 
  the smoothness of $g_0$ that $g \in C^{5+\alpha,1}$. 

After we have improved the spatial regularity to smoothness, we can use the equation
 (\ref{eq cbf in local for regula}) to improve the regularity in time to smoothness. 
 The theorem is proved.
 \end{proof}

%%%%%%%%%%%%%%%%%%%%%%%%%%%%%%%%%
%%%%%%%%%%%%%%%%%%%%%%%%%%%%%%%%%
%%%%%%%%%%%%%%%%%%%%%%%%%%%%%%%%%
\section{The backward uniqueness of CBF} \label{sec backward unique CBF}

In this section we give a proof of Theorem \ref{thm backward uniqu CBF} using 
Agmon-Nirenberg's energy method.
 This approach is used by Kotschwar in his second 
proof of the backward uniqueness of  Ricci flow (\cite{Ko16}), and later by Sun and Zhu in their proof
of the backward uniqueness of CRF (\cite{SZ18}). Because  CBF is 
a fourth order system with a pressure term, we need to combine the idea in \cite[\S 4]{Ko16} about 
high order systems with the idea in \cite{SZ18} about handling the pressure term. 

%%%%%%%%%%%%%
\subsection{Estimates of a few basic geometric quantities} 
\label{subsec est geom quant tensor}

Let $(g(t), p(t))$ and $(\tilde{g}(t), \tilde{p}(t))$ be the two solutions in Theorem
 \ref{thm backward uniqu CBF}. 
 Note that since operator $(n-1)\Delta_{g(T)} + s_0 $ is assumed
to be invertible, equality $g(T)= \tilde{g}(T)$ implies  $p(T) = \tilde{p}(T)$.  
 Below we will use  $g(t)$ as the background metric, in particular,
  connection $\nabla = \nabla_{g(t)}$, $\Delta = \Delta_{g(t)}$, and measure $d \mu = d \mu_{g(t)}$. 
Let $\tau \doteqdot T - t$ be the backward time. We define time-dependent tensors
\begin{align*}
& Y^{(0)} \doteqdot g(t) - \tilde{g}(t), \quad Y^{(1)} \doteqdot \nabla - \nabla_{\tilde{g}(t)}, 
\quad Y^{(k)} \doteqdot \nabla^{k-1} Y^{(1)} \text{ for } k \geq 2, \\
&  X^{(k)} \doteqdot \nabla^k \operatorname{Rm} - \nabla^k_{\tilde{g}(t)} 
\operatorname{Rm}_{\tilde{g}(t)} \text{ for } k \geq 0, \\
& q^{(k)} \doteqdot \nabla^k p(t)- \nabla_{\tilde{g}(t)}^k \tilde{p}(t) \text{ for } k=0,1,2,
 \quad \text{ and } q^{(k)} \doteqdot \nabla^{k-2} q^{(2)} \text{ for } k \geq 3.
\end{align*} 

The following two tensors and the estimates of their associated $L^2$-energies
 will be used to prove the backward uniqueness.
\[
X \doteqdot  X^{(0)} \oplus X^{(1)} \oplus \cdots \oplus X^{(4)} , \quad Y \doteqdot Y^{(0)}
 \oplus Y^{(1)} \oplus \cdots \oplus Y^{(4)} .
\]
\begin{remark} In the above $Y^{(k)}$ up to order $k=4$ is needed because of
 the following calculation (compare \cite[\S 3.1]{Ko16})
\[
\left ( \Delta_{\tilde{g}(t)}^2 - \Delta^2  \right ) \nabla^{a}  \operatorname{Rm}_{\tilde{g}(t)} 
 = \sum_{b=0}^4 Y^{(b)}*T_b,
\]
which is a term shown up in calculating $(\partial_{\tau} + \Delta^2) X^{(a)}$ below. Here $T_b$
 depends on $\operatorname{Rm}_{\tilde{g}(t)}$ up to its $(a+4)$th-derivatives. The order of $Y$ 
  then determines the order of $X^{(k)}$ as given above.
\end{remark}

We have the following pointwise bounds of derivatives of each components in $X$ and $Y$.

\begin{lemma} \label{lem evol eq for backward uniq}
For $a = 0,1,2,3,4$, there are constants $C_1, C_2$, and $C_3$ depending on $g(t)$ 
and $\tilde{g}(t)$ such that
\begin{align*} 
& | \partial_{\tau} Y^{(a)} |  \leq   C_1 \sum_{b=0}^{a+2} |X^{(b)}| +C_2 \sum_{b=0}^{a}
 |Y^{(b)}|  +  C_3 \sum_{b=0}^a | q^{(b)}|,  \\
& \left | ( \partial_{\tau}  + \Delta^2) X^{(a)} \right | \leq C_1 \sum_{b=0}^{a+2} |X^{(b)}| 
+C_2 \sum_{b=0}^{4} |Y^{(b)}|  +  C_3 \sum_{b=0}^{a+2} | q^{(b)}|.
\end{align*}
\end{lemma}
 
About the proof of this lemma, here we make a general comment which also applies to
the proof of Lemma \ref {lem control L2 of q(a)} below. 
In geometric analysis it is a standard practice to calculate the evolution equations of geometric 
quantities associated to the geometric flow, usually such calculations are lengthy but straightforward
and their precise schematic forms are very important for the further arguments.
In the particular case of Lemma \ref{lem evol eq for backward uniq},
 as a matter of fact, the evolution of 
 curvatures and their derivatives will be calculated in \S \ref{sect Shi conformal Bach f} for 
 other purpose, while the evolution of the differences of two metrics, the associated two
  connections, two curvatures, and two derivatives of curvatures can be calculated in a
  standard fashion. 
  
Here we skip the calculations used to prove Lemma \ref{lem evol eq for backward uniq}, instead
we  mention two key points involved in the calculations.

\noindent (I) For any tensor $W$ we hope to express  $\nabla_{\tilde{g}(t)}^{a} W$
 in terms of $\nabla^a W$ and $Y^{(b)}$ for $a \leq 4$ and $b \leq 4$, 
 in particular, for $a=2$ we have
\begin{align*}
\nabla_{\tilde{g}(t)}^2 W = \nabla^2 W + Y^{(2)} *W + Y^{(1)} * \nabla W + Y^{(1)} *W *T_1, 
\end{align*}
where $T_1$ is some tensor depending on $g(t)$ and $\tilde{g}(t)$.

\noindent  (II) A carefully calculation is needed to see that  
the upper bound of $\left | ( \partial_{\tau}  
+ \Delta^2) X^{(4)} \right | $ does not need $|Y^{(5)}|$. 
More precisely, in taking the difference of equation (\ref{eq evol nabla rm cbf}) below
 for $g(t)$ and $\tilde{g}(t)$ with $a=4$, 
 we need to bound the difference 
\[
\left | \nabla^a (T(\nabla^{2}p))  - \nabla_{\tilde{g}(t)}^a (T(\nabla_{\tilde{g}(t)}^{2}
\tilde{p})) \right |
\]
without using $|Y^{(5)}|$. For $a \leq 4$ we actually have
\begin{align*}
\nabla^{a+2} p  - \nabla_{\tilde{g}(t)}^{a+2} \tilde{p}  =& \left (\nabla^a- \nabla_{\tilde{g}(t)}^a
\right ) \nabla_{\tilde{g}(t)}^2 p +   \nabla^a \left ( \nabla^2 p - \nabla_{\tilde{g}(t)}^2 \tilde{p} 
\right ) =
  \sum_{b=0}^{a} Y^{(b)} * \tilde{T}_b  +  \sum_{b=0}^{a+2}  q^{(b)},
\end{align*}
where  $\tilde{T}_b$ depends on $g(t)$, $\tilde{g}(t)$, $p(t)$, and $\tilde{p}(t)$.
This calculation is different from the calculation in the proof of \cite[Lemma 2]{SZ18}. 
 It is this calculation which forces us to use  $q^{(b)}$ as defined above rather than 
 $\nabla^b q^{(0)}$ used in \cite{SZ18}.

\vskip .1cm
Lemma \ref{lem evol eq for backward uniq} can be summarized to  

\begin{lemma}  \label{lem evol eq for summaried backward uniq}
We have 
\begin{align}
& | (\partial_{\tau} X + \Delta^2) X | +| \partial_{\tau} Y |  \leq C\left (|Y| + |X| +|X^{(5)}|
 +|X^{(6)}|  \right )  +C \sum_{b=0}^{6} |q^{(b)}|.
\end{align}
\end{lemma}

To apply Agmon-Nirenberg's energy method we need to control $\sum_{b=0}^{6} |q^{(b)}|$
 by $|Y| +$ $ |X| +|X^{(5)}| +|X^{(6)}|$, actually controlling the $L^2$-norm
 of $\sum_{b=0}^{6} |q^{(b)}|$  is enough.
 
\begin{lemma} \label{lem control L2 of q(a)}
We have 
\begin{align*}
& \int_M \sum_{b=0}^2 | q^{(b)} |^2 d \mu \leq  C \int_M  \left ( \sum_{b=0}^2 | X^{(b)}
|^2+\sum_{c=0}^1| Y^{(c)} |^2  \right )   d \mu,   \\
&\int_M | \nabla q^{(a+2)} |^2 d \mu \leq C \int_M \left ( \sum_{b=0}^{a+2} | X^{(b)}|^2
+\sum_{c=0}^{\max \{a, 2 \}}| Y^{(c)} |^2  \right )   d \mu  \, \text{ for } a=1, 2, 3, 4.
\end{align*}
\end{lemma}
\begin{proof} (sketch) (compare to the proof of Lemma \ref{lem interpolation for p deriv via rm} below)
 The first inequality follows from $W^{2,2}$-estimate for the elliptic equation of $\left ( (n-1)
  \Delta +s_0 \right ) (p -\tilde{p})$ derived from taking the difference of second equation in
   (\ref{eq cbf eq p ellipt}) for $p$ and $\tilde{p}$ (see the proof of Lemma
    \ref{lem lip property of mathcal P op}). Here we use the assumption that
     $(n-1)\Delta_{g(t)} + s_0 $ is invertible. The second inequality follows from 
      $W^{a,2}$-estimate for the elliptic equation of $\left ( (n-1) \Delta +s_0 \right ) (\nabla^2 p
       - \nabla_{\tilde{g}(t)}^2 \tilde{p})$ which equals to an expression of $ Y^{(0)},  
       Y^{(1)},  Y^{(2)}$, and $X$.
\end{proof}

%%%%%%%%%%%%%
\subsection{Proof of Theorem \ref{thm backward uniqu CBF}} 
\label{subsec Prof thm 1.2}

Define 
\[
E(t) \doteqdot \int_M \left ( |X |^2 +|Y|^2 \right ) d \mu, \quad F(t) \doteqdot  \int_M
| \Delta X |^2  d \mu, 
\]
and their quotient  $N(t) \doteqdot \frac{F(t)}{E(t)}$. Using  Lemma \ref{lem evol eq for
 summaried backward uniq} and \ref{lem control L2 of q(a)} and by a calculation as done 
 in \cite[\S2, \S 4]{Ko16} we have
\begin{equation*}
\frac{d E}{d \tau} \leq C E +2F + C \int_M \left ( |  X^{(5)} |^2+ |  X^{(6)} |^2 \right ) d \mu.
\end{equation*}
In the calculation we use the backward and forward parabolic operators $\mathcal{L}_B =
 \partial_{\tau} -  \Delta^2$ and $\mathcal{L}_F = \partial_{\tau} +  \Delta^2$. As in 
 \cite[\S 4]{Ko16} (see also  \S \ref{subsec interpol ineq} below) we may use interpolation inequalities 
 to bound $\int_M ( |X^{(5)}|^2 + |X^{(6)}|^2 ) d \mu$ by  $\int_M |\Delta X |^2 d \mu$ 
 and $\int_M |X |^2 d \mu$, hence we get

\begin{equation} \label{eq energy L2 growth X and Y}
\frac{d E}{d \tau} \leq C (E +F) \leq C(N+1) E.
\end{equation}

Using Lemma \ref{lem control L2 of q(a)} and following the proof in \cite[\S 4]{Ko16} we can show 
\begin{align*}
& \frac{d F}{d \tau}  \geq  - C (E +F) + \frac{1}{2} \left ( \int_M | \mathcal{L}_F X |^2 d \mu 
- \int_M | \mathcal{L}_B  X |^2 d \mu  \right ), \\
& \frac{d N}{d \tau}  \geq -C(N+1) - \frac{1}{2E} \left ( \int_M | \mathcal{L}_B X |^2 d \mu 
+  \int_M | \partial_{\tau} Y |^2 d \mu  \right ).
\end{align*}
The remaining argument for the backward uniqueness  is the same as that in \cite[\S 4]{Ko16}.

%%%%%%%%%%%%%%%%%%%%%%%%%%%%%%%%%
%%%%%%%%%%%%%%%%%%%%%%%%%%%%%%%%%
%%%%%%%%%%%%%%%%%%%%%%%%%%%%%%%%%
\section{Integral version of Shi's type estimate for CBF} 
\label{sect Shi conformal Bach f}

We will prove the following  Shi's type $L^2$-estimate of derivatives of curvature 
for CBF (\ref{eq cbf eq p ellipt}). 
The estimate will be used in \S \ref{subsec appl Shi type est singular time} 
to characterize the time when CBF develops a singularity and 
in \S \ref{subsec appl Shi type compactness} to show the compactness of
a sequence of solutions of CBF.

\begin{theorem}\label{thm int Shi curvature derivative est cbf}
Let $(g(t), p(t)), \, t \in [0, T)$, be a smooth solution of CBF on closed manifold $M^n$ 
with constant scalar curvature $s_0$. Let constant $\alpha >0$. We assume that there is
 a constant $K>0$ such that the curvature of $g(t)$ and  potential function $p(t)$ satisfy
\[
\sup_{ (x,t) \in M \times \left [0, \min\{\frac{\alpha}{K}, T \} \right ]}
  \left( |\operatorname{Rm}(x,t) | +  | p (x,t)| \right) \leq K .
\]  
We also assume that operator $(n-1)\Delta_{g(t)} + s_0 $ is invertible for each $t$ with
 $\| ( (n-1) \Delta_{g(t)} + s_0)^{-1} \|_{L(C^{\alpha}, C^{2+\alpha})}\leq K$. 
 Then for any $m \in \mathbb{N}$ there exists a constant ${C}= {C}(n,s_0,
  \alpha, K, m)$ such that for all $t \in (0,  \min\{\frac{\alpha}{K}, T \}]$ we have
\begin{equation} \label{eq shi est in thm definite form}
\int_M | \nabla_{g(t)}^{m} \operatorname{Rm} (\cdot, t) |_{g(t)}^2 d \mu_{g(t)}
 \leq \frac{{C} \cdot \int_M |  \operatorname{Rm} (\cdot ,0)|_{g(0)}^2 d \mu_{g(0)} }{t^{m/2}} .
\end{equation}
\end{theorem}

The technique to derive the above estimate is standard (see, for example,
 \cite[Theorem 5.4]{St08}). We will prove the estimate through several subsections.

%%%%%%%%%%%%%%%%%%%%%%%%%%%
\subsection{Evolution equation of $\nabla_{g(t)}^a \operatorname{Rm}_{g(t)}$}
First we derive the evolution equation of $(0,4)$-curvature tensor $\operatorname{Rm} 
= \operatorname{Rm}_{g(t)}$, i.e.,  a formula for $(\partial_t  + \Delta_{g(t)}^2 ) 
\operatorname{Rm}$. We introduce a convenient notation
\begin{equation} \label{eq Bs m def notation}
B_s^k(\operatorname{Rm}) =\sum_{i_1+ \cdots +i_s =k} C_{i_1, \cdots, i_s} \nabla^{i_1} 
\operatorname{Rm} * \cdots * \nabla^{i_s} \operatorname{Rm}
\end{equation}
for some constants $C_{i_1, \cdots, i_s}$.

%%%%%%
\begin{lemma}  
Let $(g(t), p(t))$ be a  smooth solution of CBF (\ref{eq cbf eq p ellipt}) on manifold $M^n$
 with constant scalar curvature $s_0$. Then 
\begin{equation} \label{eq evol of Riema curv cbf}
\partial_t \operatorname{Rm}  + \Delta^2 \operatorname{Rm}=  B_2^2 (\operatorname{Rm}) 
+  B_3^0 (\operatorname{Rm})  +  2(n-2) p \operatorname{Rm}+ (n-2) T (\nabla^2 p),
\end{equation}
where $(0,4)$-tensor $T (\nabla^2 p)$ is defined by
\[
T (\nabla^2 p)_{ijkl} = g_{jl} \nabla_i \nabla_k p - g_{jk}   \nabla_i \nabla_l p - g_{il} 
\nabla_j \nabla_k p  + g_{ik} \nabla_j \nabla_l p .
\]
\end{lemma}

\begin{proof}
From \cite[(2.67)]{CLN6} we have
\begin{align*}
\partial_t R_{ijkl}=   & \frac{1}{2} \left (\nabla_i \nabla_k h_{jl} - \nabla_i \nabla_l h_{jk}
 - \nabla_j\nabla_k h_{il} + \nabla_j \nabla_l h_{ik} \right ) + \frac{1}{2} \left ( R_{ijkp}
 h_{pl} + R_{ijpl}h_{pk} \right ),
\end{align*}
where $h_{ij} = 2(n-2) (B_{ij}(g)+pg_{ij})$. From (\ref{eq Bach definition}) and 
scalar curvature $S_{g(t)} =s_0$ we have

\begin{align*}
\partial_t R_{ijkl}=   & \nabla_i \nabla_k \Delta R_{jl} - \nabla_i \nabla_l \Delta R_{jk}- 
\nabla_j\nabla_k \Delta R_{il} + \nabla_j \nabla_l \Delta R_{ik} + B_2^2 (
\operatorname{Rm}) +  B_3^0 (\operatorname{Rm})  \\  
&+ 2(n-2) p R_{ijkl} +(n-2) \left ( g_{jl} \nabla_i \nabla_k p -  g_{jk}\nabla_i \nabla_l p
 - g_{il} \nabla_j \nabla_k p + g_{ik} \nabla_j \nabla_l p   \right ) .
\end{align*}
Here we treat $s_0$ as a $\operatorname{Rm}$ factor. The equality (\ref{eq evol of Riema
 curv cbf}) then follows from Ricci identity (\cite[(1.30)]{CLN6}) and a well-known formula
  (\cite[(2.64)]{CLN6}).
\end{proof}

From (\ref{eq evol of Riema curv cbf}) we may derive the following by a direct calculation
which we omit.

\begin{lemma} 
Let $(g(t), p(t))$ be a smooth solution of CBF (\ref{eq cbf eq p ellipt}) on manifold $M^n$
 with constant scalar curvature $s_0$. Then
\begin{align}
\partial_t (\nabla^a  \operatorname{Rm}) + \Delta^2 (\nabla^a  \operatorname{Rm}) 
= & B_2^{2+a} (\operatorname{Rm}) + B_3^{a} (\operatorname{Rm}) + 2(n-2) p 
\nabla^a \operatorname{Rm} \notag \\
& + \sum_{b=1}^a   \nabla^b p * \nabla^{a-b}  \operatorname{Rm} 
+(n-2) \nabla^a (T(\nabla^{2}p)).  \label{eq evol nabla rm cbf}
\end{align}
\end{lemma}

%%%%%%%%%%%%%%%%%%%%%%%%%%%
\subsection{An interpolation inequality} \label{subsec interpol ineq}
We will need the following inequality (compare \cite[Corollary 5.5]{KS02} and
\cite[Lemma 10.3]{St08}).

\begin{lemma}
Let $(M^n,g)$ be a closed Riemannian manifold. 
Let integers $0 \leq i_1, \cdots, i_r \leq a \in \mathbb{N}$ and $i_1 + \cdots + i_r = 2a$.
 Then for any $\epsilon >0$ and nonnegative integers $r_1$ and $r_2$ there is a constant 
 $C =C(n, a, r, r_1, r_2, \epsilon)$ such that for any tensor $W$ of type $(r_1,r_2)$ on $M$ we have
\begin{equation} \label{eq interpol m+1 power}
\left | \int_M \nabla^{i_1} W * \cdots * \nabla^{i_r}W d \mu_g \right | \leq \epsilon \int_M|
 \nabla^{a+1} W|^2 d \mu_g + C \| W\|_{\infty}^{(a+1)(r-2)} \cdot\int_M |W|^2 d \mu_g,
\end{equation}
where $\| W\|_{\infty}$ is the $C^0$-norm.
\end{lemma}

\begin{proof} Assume $i_1, \cdots, i_l  \geq 1$ and $i_{l+1}, \cdots, i_r=0$. By the 
H\"{o}lder inequality we have
\begin{align*}
\left | \int_M \nabla^{i_1}W \ast \nabla^{i_2}W \ast \cdots \ast \nabla^{i_r}W d\mu_g \right |
 \leq  \| W\|_{\infty}^{r-l} \cdot \prod_{j=1}^l \left ( \int_M \left| \nabla^{i_j}W 
 \right |^{\frac{2a}{i_j}}d\mu_g \right)^{\frac{i_j}{2a}}.
\end{align*}
By using
\begin{equation} \label{eq Ham82 12.6}
\left ( \int_M \left | \nabla^{i_j}W  \right |^{\frac{2a}{i_j}} d\mu_g \right )^{\frac{i_j}{2a}}
 \leq C \| W \|_\infty^{1 - \frac{i_j}{a}} \left (\int_M \left |\nabla^a W \right |^2d \mu_g
  \right )^{\frac{i_j}{2a}}
\end{equation}
from \cite[Corollary 12.6]{Ha82} we have 

\begin{equation} \label{eq interpol m power de}
\left | \int_M \nabla^{i_1} W * \cdots * \nabla^{i_r}W d \mu_g \right| \leq C \| 
W\|_{\infty}^{r-2} \cdot  \int_M \big| \nabla^a W \big|^2 d\mu_g.
\end{equation}

It follows from \cite[Corollary 12.7]{Ha82} and Young's inequality that for any $A>0$ 
and $1 \leq i \leq a$ we have
\begin{equation} \label{eq interpolation key}
A \int_M |\nabla^i W|^2 d \mu_g \leq \epsilon \int_M |\nabla^{a+1} W|^2 d \mu_g + C 
A^{\frac{a+1}{a+1-i}} \int_M |W|^2 d \mu_g,
\end{equation}
where constant $C$ depends on $\epsilon, i, a$ but not on $A$. The lemma follows from
 combining (\ref{eq interpol m power de}) and (\ref{eq interpolation key}).
\end{proof}

%%%%%%%%%%%%%%%%%%%%%%
\subsection{$L^2$-estimate of the derivatives of pressure function $p$}
We have

\begin{lemma}  \label{lem interpolation for p deriv via rm}
Let $(g(t), p(t))$ be a smooth solution of CBF (\ref{eq cbf eq p ellipt}) on closed
 manifold $M^n$ with constant scalar curvature $s_0$. Then for any $\epsilon \in 
 (0, 1), \, a \in \mathbb{N}$, and each time $t$ the pressure function $p$ 
 satisfies the following energy estimate
\begin{align}  \label{eq pressure deriv estimate A} 
\int_M |\nabla^a p (t)|^2 d\mu \leq  \epsilon \int_M |\nabla^{a+2}\operatorname{Rm}|
^2d \mu+ C_1\| \operatorname{Rm} (\cdot, t) \|_{2}^2 +C_2\cdot\| p (t) \|^2_2.
\end{align}
Here constant $C_1$ depends on $a, \| \operatorname{Rm} (\cdot, t) \|_\infty \doteqdot
\sup_{x \in M} | \operatorname{Rm} (\cdot, t)|_{g(t)}$, 
and  constant $C_2$ depends on 
$a, s_0$,  $\|  \operatorname{Rm} (\cdot, t) \|_{\infty}$, and $ \| p (t)\|_{\infty}$.
\end{lemma}

\begin{proof}
Using Lemma \ref{eq div B is fourth order} we may rewrite schematically the
 second equation in (\ref{eq cbf eq p ellipt}) as
\begin{equation} \label{eq p elliptic schema version}
    \left ( -(n-1)\Delta - s_0 \right ) p = B^2_2(\operatorname{Rm}) + B_3^0( 
    \operatorname{Rm}),
\end{equation}
and hence
\begin{equation} \label{eq nabla a pf p to be}
 \left( -(n-1)\Delta - s_0 \right ) \nabla^{a-1} p = B_2^{a+1}(\operatorname{Rm}) + B_3^{a-1}
(\operatorname{Rm})+\sum_{b=1}^{a-1} \nabla^{a-b-1} \operatorname{Rm} * \nabla^b p.
\end{equation}
Multiplying the above equation by $ \nabla^{a-1} p$ and using integration by parts we get
\begin{align}
(n-1)\int_M |\nabla^a p|^2 d\mu\notag = &  s_0 \int_M |\nabla^{a-1}p|^2 d\mu
 + \int_M  B_2^{a+1}(\operatorname{Rm}) * \nabla^{a-1}p d\mu \notag \\
& + \int_M B_3^{a-1}(\operatorname{Rm})  * \nabla^{a-1}p d \mu  +
\sum_{b=1}^{a-1}  \int_M \nabla^{a-b-1} \operatorname{Rm} *\nabla^b p  *
 \nabla^{a-1}p d \mu .  \label{eq est nabla k p L2}
\end{align}
Below we estimate each terms in the right-hand-side of the above equality.

Note that by (\ref{eq interpolation key}) we have
\[
\left |s_0 \int_M |\nabla^{a-1}p|^2 d\mu \right | \leq \epsilon \int_M |\nabla^{a}p|^2
 d \mu + C|s_0|^{a}  \int_M p^2 d \mu.
\]
The next two terms can be estimated by using H\"{o}lder inequality, (\ref{eq Ham82 12.6}),
and (\ref{eq interpolation key}),

\begin{align*} 
& \left | \int_M  B_2^{a+1}(\operatorname{Rm}) * \nabla^{a-1}p d\mu\right |  
+ \left | \int_M B_3^{a-1}(\operatorname{Rm})  * \nabla^{a-1}p d \mu \right |  \\
\leq~ & \epsilon \int_M |\nabla^{a+2}\operatorname{Rm}|^2 d \mu + C( \left \| 
\operatorname{Rm} \right \|_{\infty}^{2(a+2)}  + \left\|  \operatorname{Rm}
 \right \|_{\infty}^{\frac{4(a+2)}{3}} )\int_M \left | \operatorname{Rm} \right |^{2} d \mu  \\
 & +\epsilon \int_M |\nabla^{a}p|^2 d \mu + C \int_M p^2 d \mu.
\end{align*}

\vskip .1cm
The term $\int_M \nabla^{a-b-1} \operatorname{Rm} \ast \nabla^b p  \ast \nabla^{a-1}p 
d \mu$ can be estimated by using H\"{o}lder inequality, (\ref{eq Ham82 12.6}), and 
Young's inequality,
\begin{align*}
& \left | \int_M \nabla^{a-b-1} \operatorname{Rm} \ast \nabla^b p  \ast \nabla^{a-1}p
 d \mu \right | \\\leq & \left ( \int_M \left | \nabla^{a-b-1} \operatorname{Rm}  
 \right |^{\frac{2(a-1)}{a-b-1}} d \mu \right )^{\frac{a-b-1}{2(a-1)}} \cdot \left ( 
 \int_M \left | \nabla^b p \right |^{\frac{2(a-1)}{b}} d \mu \right )^{\frac{b}{2(a-1)}}  
\cdot \left ( \int_M \left | \nabla^{a-1}p \right |^{2} d \mu \right )^{\frac{1}{2}} \\
\leq & C \left \|  \operatorname{Rm}  \right \|_{\infty}^{\frac{2b}{a-b-1}} \int_M 
\left | \nabla^{a-1}\operatorname{Rm} \right |^{2} d \mu + C\| p\|_{\infty}^{
\frac{2(a-b-1)}{a+b-1}} \int_M \left | \nabla^{a-1}p \right |^2 d\mu  \\
\leq & \epsilon \int_M|\nabla^{a+2}\operatorname{Rm}|^2 d \mu + C \left \|
 \operatorname{Rm} \right\|_{\infty}^{\frac{2b(a+2)}{3(a-b-1)}}\int_M \left| 
 \operatorname{Rm} \right |^{2} d \mu \\
 & + \epsilon \int_M |\nabla^{a}p|^2 d \mu
  + C \| p\|_{\infty}^{\frac{2a(a-b-1)}{a+b-1}} \int_M p^2 d \mu,
\end{align*}
where we have used (\ref{eq interpolation key}) to get the last inequality.
Hence the lemma is proved.
\end{proof}

%%%%%%
\vskip .1cm
Next we  bound $\int_M p^2 d \mu$. 

\begin{lemma} \label{lem est for p L2 by Rm}
Let $(g(t), p(t))$ be a smooth solution of CBF (\ref{eq cbf eq p ellipt}) on closed
 manifold $M^n$ with constant scalar curvature $s_0$. Assume that for each $t$ 
 operator $(n-1) \Delta_{g(t)} + s_0$ is invertible, then for any $\epsilon \in (0,1)$
  there exists a positive constant $C$ depending on $n$ and $\epsilon$ such that 
  for each time $t$ the pressure function satisfies
\[
\|p (t) \|_{2}^2 \leq \epsilon \int_M \left | \nabla^{3} \operatorname{Rm}\right |^2
 d\mu + C \left ( \left \| \operatorname{Rm}  (\cdot, t) \right \|_\infty^4 +\left \| \operatorname{Rm}
  (\cdot, t) \right \|_\infty^6  \right )  \| \operatorname{Rm} \|_2^2.
\]
\end{lemma}

\begin{proof}
Since the operator $(n-1)\Delta +s_0$ is invertible, by equation
 (\ref {eq p elliptic schema version}) and the standard elliptic theory we have the following
\begin{align*}
\|p (t) \|_{2}^2 \leq &  C  \left ( \|  B^2_2(\operatorname{Rm}) \|_{2}^2 + 
\| B_3^0( \operatorname{Rm}) \|_{2}^2  \right ) \\
\leq & C \left( \left \| \operatorname{Rm}  (\cdot, t) \right \|_\infty^2 \int_M \left | \nabla^2
 \operatorname{Rm} \right |^2 d\mu+ \left \| \operatorname{Rm}  (\cdot, t) \right \|_\infty^4  
 \int_M \left | \operatorname{Rm}\right |^2 d\mu  \right ).
\end{align*}
The lemma now follows from applying (\ref{eq interpolation key}) to 
$\left \| \operatorname{Rm} (\cdot, t) \right \|_\infty^2 \int_M \left | \nabla^2 \operatorname{Rm}
 \right |^2 d\mu$.
\end{proof}

\begin{remark}
Note that this is the only place where we use the assumption that $(n-1) \Delta + s_0$ 
is invertible in the proof of Theorem \ref{thm int Shi curvature derivative est cbf}. 
We may replace the assumption by a upper bound of $\|p (t) \|_{2}$. 
As a consequence we may replace condition (ii) in Theorem
 \ref{thm when extension CBF possible} below by such a condition, and condition (ii) in 
 Theorem \ref{thm compactness sequence cbf} below by condition $\|p_i (t) \|_{2} 
 \leq K$ for all $i$ and $t$.
\end{remark}

%%%%%%%%%%%%%%%%%%%%%%%%%%%
\subsection{Differential inequality about $ \frac{d}{dt} \int_M \left | \nabla^a 
\operatorname{Rm} \right |^2 d \mu$} We have

\begin{proposition} \label{prop curv derv L2 evol ineq}
Let $(g(t), p(t))$ be a smooth solution of CBF (\ref{eq cbf eq p ellipt}) on closed
 manifold $M^n$ with constant scalar curvature $s_0$. 
 Then for any $\epsilon \in (0,1)$ and any $a \in \mathbb{N}$ there 
 exist constants $C_1$ and $C_2$ depending
  on $n, \epsilon, a, s_0, \, \| \operatorname{Rm} (\cdot, t) \|_{\infty}$, and $\| p (t) \|_{\infty}$ such that 
\begin{align} 
\frac{d}{dt} \int_M |\nabla^a \operatorname{Rm}|^2 d \mu  + (2- \epsilon) \int_M | 
\nabla^{a+2} \operatorname{Rm}|^2 d \mu \leq   C_1 \| 
\operatorname{Rm} \|_2^2 +C_2\|p \|_2^2. \label{eq L2 estimate new}
\end{align}
When $a=0$ we have 
\begin{equation} 
\frac{d}{dt} \int_M |\operatorname{Rm}|^2 d \mu + 2\int_M | \nabla^{2}
 \operatorname{Rm}|^2 d \mu \leq  C \left (1+  \| \operatorname{Rm} (\cdot, t) \|_{\infty}^2 
  +\|p (t) \|_{\infty} \right ) \cdot \| \operatorname{Rm} \|_2^2,  \label{eq rm no deri evol L2}
\end{equation}
where $C$ is a constant depending only on $n$.
\end{proposition}

\begin{proof}
From  (\ref{eq evol nabla rm cbf}) we have 
\begin{align}
\frac{d}{dt} \int_M |\nabla^a \operatorname{Rm} |^2 d \mu = & -2 \int_M  \left
| \nabla^{a+2}  \operatorname{Rm}  \right |^2 d \mu + \int_M   B_3^{2+2a}
 (\operatorname{Rm}) d \mu + \int_M B_4^{2a} (\operatorname{Rm})   d \mu \notag \\
& + \int_M  \nabla^a \operatorname{Rm} * \nabla^a (T(\nabla^{2}p)) d \mu + 
\sum_{b =0}^{a}  \int_M B_2^{2a-b}(\operatorname{Rm})*\nabla^b p d \mu,
 \label{eq wejust added it}
\end{align}
where we have used the following integration by parts to get the first term
 on the right-hand-side above
\begin{align*}
\int_M \langle\nabla^a \operatorname{Rm} ,  \Delta^2 (\nabla^a  \operatorname{Rm})
 \rangle d \mu=\int_M \left | \nabla^{a+2}  \operatorname{Rm}  \right |^2 d \mu 
 +\int_M  B_3^{2+2a} (\operatorname{Rm}) d \mu . 
\end{align*}

For the terms behind the summation sign in (\ref{eq wejust added it})
 and for $1 \leq b \leq a-1$ we use H\"{o}lder inequality and
 (\ref{eq Ham82 12.6}) to get
\begin{align*}
&\left | \int_M \nabla^a\operatorname{Rm}*\nabla^b p * \nabla^{a-b}
 \operatorname{Rm} d \mu  \right |\\
\leq & \left ( \int_M  \left | \nabla^a \operatorname{Rm} \right |^2d \mu 
\right )^{\frac{1}{2}} \cdot \left ( \int_M \left | \nabla^{a-b}\operatorname{Rm} 
\right |^{\frac{2a}{a-b}}  d \mu  \right )^{\frac{a-b}{2a}} \cdot \left ( \int_M \left | 
\nabla^b p \right |^{\frac{2a}{b}} d \mu \right )^{\frac{b}{2a}} \\
\leq & C  \|\operatorname{Rm} \|_\infty ^{\frac{b}{a}} \left ( \int_M  \left | 
\nabla^a\operatorname{Rm} \right |^2 d \mu \right )^{\frac{2a-b}{2a}} \cdot 
 \| p\|_{\infty}^{\frac{a-b}{a}}\left ( \int_M \left | \nabla^a p\right |^2  d \mu
   \right )^{\frac{b}{2a}} \\
\leq & C \|\operatorname{Rm} \|_\infty ^{\frac{2b}{2a -b}} \cdot\| p
 \|_{\infty}^{\frac{2(a-b)}{2a-b}}   \int_M \left | \nabla^{a} \operatorname{Rm} 
  \right |^2 d \mu +  C \int \left | \nabla^{a}p \right |^2 d \mu \\
\leq & 2 \epsilon \int_M \left | \nabla^{a+2} \operatorname{Rm}  \right |^2 d \mu 
+ C(\epsilon) \|\operatorname{Rm} \|_\infty ^{\frac{b(a+2)}{2a -b}} \cdot  \|
 p\|_{\infty}^{\frac{(a+2)(a-b)}{2a-b}}   \int_M \left |\operatorname{Rm}   
   \right |^2 d \mu  \\
& +C_1(\epsilon, a, \|\operatorname{Rm}  \|_{\infty} ) \int_M \left | 
\operatorname{Rm} \right |^2 d \mu  + C_2(\epsilon, a, s_0,
 \|\operatorname{Rm}\|_{\infty},  \| p\|_{\infty}  ) \int_M p^2 d \mu,
\end{align*}
where we have used Young's inequality  to get the second to last inequality, 
and (\ref{eq interpolation key}) and (\ref{eq pressure deriv estimate A}) to get the last inequality. 
When $b=a$ or $b=0$ the above estimate still holds and the proof is actually simpler.

By using interpolation inequality (\ref{eq interpol m+1 power}) and  integration by parts  we have
\begin{align*}
& \left |\int_M  B_3^{2+2a}
 (\operatorname{Rm}) d \mu \right| +  \left | \int_M B_4^{2a} (\operatorname{Rm}) d \mu \right | 
 \leq   \epsilon \int_M \left | \nabla^{a+2}
  \operatorname{Rm}  \right |^2 d \mu +C(\epsilon)  \|\operatorname{Rm} 
  \|_{\infty}^{a+2} \int_M \left | \operatorname{Rm}  \right |^2 d \mu.
  \end{align*}
By using  integration by parts  and  (\ref{eq pressure deriv estimate A})  
we have 
\begin{align*}
\left | \int_M \nabla^{a} \operatorname{Rm} * \nabla^{a+2} p \, d \mu \right | 
 \leq & \epsilon  \int_M \left | \nabla^{a+2}
 \operatorname{Rm} \right |^2 d \mu +  (4 \epsilon)^{-1} \int \left | \nabla^{a}p \right |^2 d \mu \\
\leq  & \epsilon \int_M \left | \nabla^{a+2} \operatorname{Rm}  \right |^2 d \mu +
C_1(\epsilon, a, \|\operatorname{Rm}\|_{\infty} ) \int_M \left | \operatorname{Rm} \right |^2 d \mu  \\
& +C_2(\epsilon, a,s_0, \|\operatorname{Rm} \|_{\infty},  \| p\|_{\infty} ) \int_M p^2 d \mu.
\end{align*}

Putting together  all calculations above we obtain estimate (\ref{eq L2 estimate new}).
 A careful check of the above calculation will produce the inequality for $a=0$.
\end{proof}

%%%%%%%%%%%%%%%%%%%%%%%%%%%
\subsection{Proof of Theorem \ref{thm int Shi curvature derivative est cbf}}
We use the idea in the proof of \cite[Theorem 5.4]{St08}.
From  Proposition \ref{prop curv derv L2 evol ineq} with $\epsilon = \frac{1}{4}$ 
we have the following. For $a \geq 2$
\begin{align*}
\frac{d}{dt} \int_M |\nabla^a \operatorname{Rm}|^2 d \mu + \frac{7}{4} \int_M
 | \nabla^{a+2} \operatorname{Rm}|^2 d \mu  \leq & C_1  \| \operatorname{Rm} 
 \|^2_2 +C_2  \| p  \|^2_2. 
\end{align*}
Let $\bar{T} \doteqdot  \min \{ \frac{\alpha}{K}, T \}$. Below $t \in (0, \bar{T}]$. 
Note that by applying Lemma \ref{lem est for p L2 by Rm} to bound the last term in 
the above inequality and using (\ref{eq interpolation key})
to bound $ \int_M \left | \nabla^{3} \operatorname{Rm}\right |^2
 d\mu $ by $ \frac{1}{4} \int_M | \nabla^{4} \operatorname{Rm}|^2 d \mu$ plus
something else,
 we get
\begin{align}
& \frac{d}{dt} \int_M |\nabla^a \operatorname{Rm}|^2 d \mu  + \frac{3}{2} \int_M 
| \nabla^{a+2} \operatorname{Rm}|^2 d \mu \leq  \frac{1}{2^{a+1}\bar{T}^{a-1}}  
\| \nabla^4\operatorname{Rm}\|^2_2+C_3(a,\bar{T})  \| \operatorname{Rm} \|^2_2 , 
 \label{eq shi est tem 12}
\end{align}
where constant $C_3(a,\bar{T})$ depends on $C_1$ and $C_2$. As a special case but using
different factor in front of $\| \nabla^4\operatorname{Rm}\|^2_2$, we have
 \begin{align}  \label{eq shi est tem 12 a=2}
& \frac{d}{dt} \int_M |\nabla^2 \operatorname{Rm}|^2 d \mu  + \int_M 
| \nabla^{4} \operatorname{Rm}|^2 d \mu \leq -  \frac{1}{2}  
\| \nabla^4\operatorname{Rm}\|^2_2+C_4( \bar{T})  \| \operatorname{Rm} \|^2_2.
\end{align}

 For $ a \geq 1$ we define function 
\begin{equation*}
    f_a(t) \doteqdot \sum_{j=0}^{a} \frac{t^{j}}{j!}  \int_M |\nabla^{2 j}\operatorname{Rm}|^2 d \mu.
\end{equation*}
Using (\ref{eq shi est tem 12}), (\ref{eq shi est tem 12 a=2}), 
and (\ref{eq rm no deri evol L2}) we compute
\begin{align}
\frac{d}{dt} f_a(t) =& \sum_{j=1}^{a} \frac{t^{j}}{j!}\left (  \frac{d}{dt} \int_M
 |\nabla^{2j}\operatorname{Rm}|^2 d \mu + \int_M |\nabla^{2j+2}\operatorname{Rm}|^2 
 d \mu  \right )  \notag   \\
& + \left ( \frac{d}{dt} \int_M | \operatorname{Rm}|^2 d \mu +  \int_M |\nabla^{2}
\operatorname{Rm}|^2 d \mu \right )  - \frac{t^{a}}{a!}  \int_M |\nabla^{2a+2}
\operatorname{Rm}|^2 d \mu \notag  \\
\leq  &  \| \operatorname{Rm} \|^2_2 \sum_{j=2}^{a} \left (\frac{t^{j}}{j!} C_3(2j,
 \bar{T}) \right )+ \|\nabla^4 \operatorname{Rm}\|_2^2  \left(\sum_{j=2}^a \left 
 (\frac{t^{j}}{j!} \cdot \frac{1}{2^{2 j+1}T^{2j-1}}  \right ) -\frac{t}{2}  \right ) \notag   \\
& + C_4(\bar{T} ) \cdot  t  \| 
\operatorname{Rm} \|^2_2 +  C  ( 1+ \| \operatorname{Rm} \|_{\infty}^2 
+\|p\|_{\infty} )  \cdot  \| \operatorname{Rm} \|^2_2  \notag    \\
 \leq &  C(n, s_0, \bar{T}, \| \operatorname{Rm} \|_{\infty}^2 ,
 \|p\|_{\infty} , a)  \cdot  \| \operatorname{Rm} \|^2_2.  \label{eq fk ineq for shi L2 name}
\end{align}
This implies (\ref{eq shi est in thm definite form}) when $m$ is even.

For (\ref{eq shi est in thm definite form}) with odd $m$ we use the following interpolation to get it,
\begin{align*}
\int_M |\nabla^{2a+1} \operatorname{Rm}|^2 d \mu \leq &\left (\int_M |\nabla^{2a} 
\operatorname{Rm}|^2 d \mu \right )^{1/2} \left ( \int_M |\nabla^{2a+2}
 \operatorname{Rm}|^2 d \mu \right  )^{1/2} 
\end{align*}
This finishes the proof of Theorem \ref{thm int Shi curvature derivative est cbf}.

\vskip .2cm
We end this section with  two direct applications of 
Theorem \ref{thm int Shi curvature derivative est cbf}.

%%%%%%%%%%%%%%%%%%%%%%%%%
%%%%%%%%%%%%%%%%%%%%%%%%%
\subsection{Characterizing finite time singularities} \label{subsec appl Shi type est singular time}
Recall that for a closed Riemannian manifold $(M^n,g)$ the Sobolev constant $C_S(g)$ 
is the smallest constant $C$ such that for any function $u \in C^1(M)$ we have
\begin{equation}
    \| u \|_{\frac{2n}{n-2}} \leq C \left (\| \nabla u \|_{2} + \operatorname{Vol}_g^{-\frac{1}{n}}
     \| u \|_{2} \right ),
\end{equation}
where $ \| u \|_{p}$ is the $L^p$-norm and  $\operatorname{Vol}_g$ is the volume of $M$. 
The first application of Theorem \ref{thm int Shi curvature derivative est cbf} is  

\begin{theorem} \label{thm when extension CBF possible}
Let $(g(t), p(t)), \, t \in [0, T)$, be a smooth solution of CBF (\ref{eq cbf eq p ellipt}) 
on closed manifold $M^n$ with constant scalar curvature $s_0$. We assume that there
 is a constant $K>0$ such that the curvature of $g(t)$ and  potential function $p(t)$ 
 satisfy the following conditions

\noindent (i) $\sup_{ (x,t) \in M \times  [0, T)} \left (  |\operatorname{Rm}(x, t) |_{g(t)}
+  | p (x,t)|  \right ) \leq K$,

\noindent (ii) the operator norm $\| ( (n-1) \Delta_{g(t)} + s_0)^{-1} \|_{L(C^{\alpha},
C^{2+\alpha})}\leq K$   for  $t \in [0, T)$, and

\noindent (iii) the Sobolev constant $C_S(g(t)) \leq K$  for  $t \in [0, T)$.

\noindent Then $(g(t), p(t))$ can be extended to a smooth solution of CBF on $[0,T+ \delta]$ 
for some $\delta >0$.  
\end{theorem}

\begin{proof} From assumption (i), (ii), and Theorem 
\ref{thm int Shi curvature derivative est cbf} we get uniform bounds on $L^2$-norm of
 any $m$-th derivatives of the curvature, $\| \nabla^{m} \operatorname{Rm}(\cdot, t)
  \|_{L^2}\leq C(m,T,K)$ for $t \in [T/2, T)$. Then using assumption (iii) 
about the Sobolev constant  and arguing
   as in the proof of \cite[Theorem 6.2]{St08} we get uniform $C^k$-norm  bounds 
   of curvatures of $g(t)$ for $t \in [T/2,T)$. 
   
 Note that the uniform $C^4$-norm bounds
    imply the uniform equivalence of metric $g(t), \, t \in [T/2,T)$, with $g(T/2)$,
  hence we have a uniform lower bound of injectivity radius $\operatorname{inj}_{g(t)} 
  \geq \iota > 0$.
 By the Cheeger-Gromov compactness theorem of Riemannan manifolds
      we conclude that $g(t)$ converges smoothly to a metric called $g(T)$ as $t \to T^-$
       (no diffeomorphisms needed). Then we can use assumption (ii) and Theorem
        \ref{thm main short time exist cbf} to extend metric $g(T)$  to a solution 
        $g(t)$ of CBF  for $t \in [T, T +\delta ]$.
\end{proof}

%%%%%%%%%%%%%%%%%%%%%%%%%%%%%%%%%
%%%%%%%%%%%%%%%%%%%%%%%%%%%%%%%%%
%%%%%%%%%%%%%%%%%%%%%%%%%%%%%%%%%
\subsection{Compactness theorem for CBF}  \label{subsec appl Shi type compactness}
 The proof of the following theorem is standard and is also similar to that of 
\cite[Theorem 7.1]{St08}, we omit it.

\begin{theorem} \label{thm compactness sequence cbf}
Let $\{ (g_i(t), p_i(t))\}, \, t \in (\alpha, \omega)$, be a family of smooth solutions of CBF
 (\ref{eq cbf eq p ellipt})  on closed manifolds $M_i^n$ with constant scalar curvature $s_{0i}$,
  where $-\infty \leq \alpha <0 < \omega \leq \infty$. Let $\{q_i \in M_i \}$ be a sequence 
  of points. We assume that there is a constant $K>0$ such that for each $i$ the curvature of $g_i(t)$ 
  and potential function $p_i(t)$ satisfy the following conditions

\noindent (i) $\sup_{ (x,t) \in M_i \times (\alpha, \omega)} \left ( |\operatorname{Rm}_{g_i}(x,t)
 |_{g_i(t)}  + | p_i (x,t)| \right) \leq K$,

\noindent (ii) the operator norm $\| ( (n-1) \Delta_{g_i(t)} + s_{0i})^{-1} \|_{L(C^{\alpha}, 
C^{2+\alpha})}\leq K$   for  $t \in (\alpha, \omega)$,

\noindent (iii) the Sobolev constant  $C_S(g_i(t)) \leq K$ for $ t \in (\alpha, \omega)$, and

\noindent (iv) $\lim_{t \to \alpha} \int_{M_i} |  \operatorname{Rm}_{g_i}(\cdot, t)|_{g_i(t)}^2
 d \mu_{g_i(t)}  \leq K$.

\noindent Then sequence $\{(M_i,g_i(t), p_i(t), q_i) \}$ sub-converges in pointed 
Cheeger-Gromov $C^{\infty}$-topology to a complete solution $(M_{\infty}^n,
 g_{\infty}(t), p_{\infty}(t), q_{\infty}), \, t \in (\alpha, \omega)$ of CBF (\ref{eq cbf eq p ellipt}).  
\end{theorem}

%%%%%%%%%%%%%%%%%%%%%%%%%%%
%%%%%%%%%%%%%%%%%%%%%%%%%%%
%%%%%%%%%%%%%%%%%%%%%%%%%%%
\section{Local integral and pointwise Shi's type estimate} 
\label{sect Improved Shi conf Bach flow}

In this section we prove a local version of the integral Shi's type estimate (Theorem 
\ref{thm improved Shi curvature derivative conseq}). 
At the end we give a proof of  Theorem \ref{thm ptwise Shi est for CBF-intro}.
 The basic ideas of the two proofs are from \cite[Theorem 4.4]{St13}  and  \cite[\S 5]{St13}, 
 respectively.

%%%%%%%%%%%%%%%%%%%%%%%%%%%
\subsection{A property about cutoff functions and localized interpolation inequalities}
 \label{subsec cutoff local interpol}

We will need the following when we use cutoff functions to localize later.
 
\begin{lemma} \label{lem cut off funct deriv est}
Let $(M^n, g(t)), \, t \in [0,T]$, be a smooth family of Riemannian manifolds. 
Then there  are constants 
\begin{align*}
C_1=C_1 \left ( \sup_{t \in [0,T]}  | \partial_t g |_{g(t)} ,  T \right ) \,
\text{ and  } \,\, C_2 = C_2 \left (
 \sup_{t \in [0,T]} (  | \partial_t g |_{g(t)} + |\nabla \partial_t g|_{g(t)}), T \right )
\end{align*}
such that  for any function $\eta \in C^{\infty}(M)$ we have pointwise-estimates
\begin{align*}
| \nabla \eta |_{g(t)} \leq C_1  |\nabla \eta |_{g(0)}  \, \text{ and } \,\, |\nabla^2
 \eta |_{g(t)} \leq C_2 \left ( |\nabla \eta |_{g(0)} t+ |\nabla^2 \eta |_{g(0)} \right )
  \, \text{for }  t \in [0, T]. 
\end{align*}
\end{lemma}

\begin{proof} The first estimate follows from
\[
\partial_t |\nabla \eta |^2 =\partial_t g_{ij}(t) \partial_i \eta  \partial_j \eta 
\leq |\partial_t g| \cdot |\nabla \eta|^2.
\] 
For the second estimate we express the Hessian in local coordinates as
\[
(\nabla^2 \eta )_{i j} (t) = \partial_i \partial_j \eta -\Gamma_{ij}^k (t) \partial_k \eta,
\]
and we compute the derivative
\begin{align*}
\partial_t |\nabla^2 \eta |^2= & \nabla \partial_t g * \nabla \eta *\nabla^2 \eta 
+ \partial_t g * \nabla^2 \eta  * \nabla^2 \eta    \\
\leq & C(n)  \sup_{t \in [0,T]}  |\nabla \partial_t g | \cdot \sup_{t \in [0,T]}| \nabla \eta
 | \cdot  | \nabla^2 \eta | + \sup_{t \in [0,T]} | \partial_t g | \cdot  |
  \nabla^2 \eta |^2.
\end{align*}
Using the first estimate and applying the Gronwall-type inequality we get the second estimate.
\end{proof}

\vskip .1cm
The following local interpolation inequalities are simple consequence from \cite[\S 5]{KS02}
 (see also \cite[\S 10]{St08}) which are analog of  (\ref{eq interpolation key}) and 
 (\ref{eq interpol m+1 power}).

\begin{lemma} \label{cor weight interpol ineq}
Let $(M^n,g)$ be a Riemannian manifold and let $\eta$ be a $C^1$ function
which satisfies $0 \leq \eta \leq 1$ and $|\nabla \eta | \leq \Lambda$. 
We assume that set $\{ x \in M, \, \eta(x) >0 \}$ is precompact in $M$. 
Let $a \in \mathbb{N}$ and let $W$ be any $C^{a+1}$-tensor of type $(r_1,r_2)$. 
Then 

\noindent (i) For $i=1, 2, \cdots, a$, and constants $A>0, \, \epsilon >0, \, s \geq 2a$, 
we have

\begin{equation} \label{eq interpolation key with eta}
A \int_M |\nabla^i W|^2 \eta^s d \mu \leq \epsilon \int_M |\nabla^{a+1} W|^2 
\eta^{s+2a+2 -2i} d \mu + C A^{\frac{a+1}{a+1-i}}  \|W\|^2_{2, \eta >0} ,  
\end{equation}
where constant $C =C(n,a, r_1,r_2, s, \Lambda, \epsilon)$. 

\noindent (ii) For $0 \leq i_1, \cdots, i_r \leq a \in \mathbb{N}$ with $i_1 +
 \cdots + i_r = 2a$, and   constants  $\epsilon >0, s \geq 2a$,  we have

\begin{equation} \label{eq interpol m only power m+1 power}
\left | \int_M \nabla^{i_1} W * \cdots * \nabla^{i_r}W \eta^s d \mu_g \right |
 \leq  \epsilon \int_M | \nabla^{a+1} W|^2 \eta^{s+2} d \mu + C \| 
 W \|_{\infty}^{(a+1)(r-2)} \| W\|^{2}_{2,\eta >0}, 
\end{equation}
where constant $C=C(n, a, r, r_1, r_2, s, \Lambda, \epsilon)$. 
\end{lemma}

%%%%%%%%%%%%%%%%%%%%%%%%%%%
%%%%%%%%%%%%%%%%%%%%%%%%%%%
\subsection{Localized integral version of Shi's type estimate} 
\label{subsec Localized integral version of Shi}

Let $(g(t), p(t))_{ t \in [0, T)}$ 
be a local smooth solution of CBF (\ref{eq cbf eq p ellipt}) on manifold $M^n$ 
with constant scalar curvature $s_0$, and let $\eta$ be a $C^1$ function
which satisfies $0 \leq \eta \leq 1$ and $|\nabla \eta |_{g(t)} \leq \Lambda$. 
In this subsectione we further assume that $\{ x \in M, \, \eta(x) >0 \}$ is precompact in $M$. 
The logic steps of the proof of Theorem \ref{thm improved Shi curvature derivative conseq}
 is similar to that in \S \ref{sect Shi conformal Bach f}.

\vskip .1cm
First we prove a localized version of (\ref{eq pressure deriv estimate A}).

\begin{lemma} \label{lem 4.6 localized}
For any $\epsilon \in (0, 1), \, a \in \mathbb{N}$, and $s \geq 2a$ we have
 for each time $t$
\begin{equation} \label{eq pressure deriv estimate eta s ver}
\int_M \left |  \nabla^a p (t) \right |^2 \eta^s d \mu \leq  \epsilon \int_M |\nabla^{a+2}
 \operatorname{Rm}|^2 \eta^{s+4} d\mu + C_1 \| \operatorname{Rm} \|_{2, 
 \eta >0}^2 + C_2 \| p  \|^2_{2, \eta >0}, 
\end{equation}
where constant $C_1$ depends on $n, a, s_0, \| \operatorname{Rm} (\cdot, t) \|_{\infty,
 \eta >0}, s, \Lambda$ and  constant $C_2$ further depends on $ \| p(t) \|_{\infty, \eta >0}$.
 
\end{lemma}

\begin{proof} 
We start with multiplying equation (\ref{eq nabla a pf p to be}) by $\nabla^{a-1} p 
\cdot \eta^s$ and using integration by parts. 
Note that the $\nabla \eta $ term is bounded  $\Lambda$. 
To finish the estimate we need to use  interpolation inequalities in 
Lemma \ref{cor weight interpol ineq},
and to be careful with the power of $\eta$ when applying H\"{o}lder inequality.
\end{proof}

%%%%%%%%%
\vskip .2cm
Next we establish a differential inequality about $\frac{d}{dt} \int_M |\nabla^a \operatorname{Rm} 
|^2 \eta^s d \mu$.

\begin{proposition} \label{prop curv derv L2 evol ineq cutoff}
We  assume that  $|\nabla^2 \eta (x) |_{g(t)}\leq \Lambda$. 
For $a \in \mathbb{N}$ and  $s \geq 2a$ we have
\begin{align} 
\frac{d}{dt} \int_M |\nabla^a \operatorname{Rm}|^2  \eta^s  d \mu + (2- \epsilon)
 \int_M | \nabla^{a+2} \operatorname{Rm}|^2  \eta^{s}  d \mu \leq  C_1 \|
  \operatorname{Rm} \|_{2, \eta >0}^2 + C_2 \| p \|^2_{2, \eta >0},
   \label{eq L2 estimate new cutoff}
\end{align}
where constants $C_1$ and $C_2$ are as in Lemma \ref{lem 4.6 localized}.
\end{proposition}

\begin{proof} 
The proof is similar to Proposition \ref{prop curv derv L2 evol ineq}, here we only give a
rough sketch. From equation (\ref{eq evol nabla rm cbf}) and that  $\eta$ is independent of $t$, 
we have
\begin{align*}
& \frac{d}{dt} \int_M |\nabla^a \operatorname{Rm} |^2 \eta^s d \mu   + 2 \int_M 
\left | \nabla^{a+2}  \operatorname{Rm}  \right |^2 \eta^s  d \mu\\
=& \int_M B_3^{2+2a} (\operatorname{Rm}) \eta^s  d \mu   + \int_M B_4^{2a}
 (\operatorname{Rm}) \eta^s d \mu + \int_M  \nabla \eta  * B_2^{2a+3}
  (\operatorname{Rm})\eta^{s-1}  d \mu \\
& + \int_M  (\nabla \eta)^2 * B_2^{2a+2} (\operatorname{Rm}) \eta^{s-2}
  d \mu + \int_M  \nabla^2 \eta *  B_2^{2a+2} (\operatorname{Rm}) \eta^{s-1} 
   d \mu +  \int_M p  |\nabla^a \operatorname{Rm}|^2 \eta^s  d \mu\\
&+ \sum_{b =1}^{a}   \int_M \nabla^a \operatorname{Rm}*\nabla^b p * 
\nabla^{a-b} \operatorname{Rm} \eta^s  d \mu\notag + \int_M  \nabla^a 
\operatorname{Rm} * \nabla^a (T(\nabla^{2}p)) \eta^s  d \mu,
\end{align*}
where we have applied integration by parts to $\int_M \langle\nabla^a \operatorname{Rm} , 
\Delta^2 (\nabla^a  \operatorname{Rm}) \rangle \eta^s  d \mu$.

Let 
\[
L_{\epsilon} \doteqdot \epsilon\int_M | \nabla^{a+2} \operatorname{Rm}|^2  \eta^{s}
  d \mu+  C_1 \| \operatorname{Rm} \|_{2,\eta >0}^2 + C_2 \| p \|^2_{2, \eta >0}.
\]
By the localized interpolation inequalities in Lemma \ref{cor weight interpol ineq},
 $p(t) \leq  \| p(t) \|_{\infty, \eta >0}$, and $|\nabla \eta | \leq \Lambda$, 
we conclude that all of the terms on the right-hand-side of the above display
 are bounded by $(\Lambda + 1) L_{\epsilon}$ 
except those terms which contain function $p$. For them,
by Cauchy-Schwarz inequality and the localized interpolation inequalities, we have
\begin{align*}
\int_M \left ( \left | \ \nabla^a \operatorname{Rm} * \nabla^{a+2} p \eta^s \right | + 
\left | \nabla^a \operatorname{Rm}*\nabla^b p * \nabla^{a-b} \operatorname{Rm} 
\eta^s  \right | \right )d \mu 
\leq  L_{\epsilon} +   C(\epsilon, \Lambda) \int_M \left | \nabla^{a}p
 \right |^2 \eta^s d \mu.
\end{align*}
 Note that by Lemma \ref{lem 4.6 localized} we have $\int_M \left | \nabla^{a}p
  \right |^2 \eta^s d \mu \leq L_{\epsilon}$, hence inequality (\ref{eq L2 estimate 
  new cutoff}) follows from combining the inequalities above.
  \end{proof}

%%%%
\vskip .2cm
Finally we give a proof of

\begin{theorem} \label{thm improved Shi curvature derivative conseq}
Let $(M^n, g(t), p(t))_{t \in [0, T ]}$ be a local smooth solution of CBF (\ref{eq cbf eq p ellipt}). 
Fix $r > 0$ and $q \in M$ we assume ball $B_{g(0)}(q, 2r ) $ is precompact 
in $M$ and 
\begin{equation} \label{eq assum partial t derv g RM p}
\sup_{B_{g(0)}(q, 2r) \times [0,T ]} \left ( |\operatorname{Rm}(\cdot,\cdot)| + |p(\cdot)| 
+|\partial_t g| + |\nabla \partial_t g | \right ) \leq K
\end{equation}
for some constant $K>0$. Then for $m \in \mathbb{N}$ we have
\[
\int_{ B_{g(0)}(q,r )}| \nabla^m \operatorname{Rm}(\cdot, t) |^2 d \mu
  \leq \frac{C} {t^{m/2} },
\]

where $C$ is a constant which depends on  $n, m, r, g(0), T, K$, and 
\[
\sup_{t \in [0,T ]}  \left ( \| \operatorname{Rm} (\cdot, t) \|^2_{L^2(B_{g(0 )}(p,2r ))}
+  \| p (t) \|^2_{L^2(B_{g(0 )}(p,2r ))} \right ).
\]
\end{theorem}

\begin{proof} First we show the estimate for even $m = 2a$.
Let $\eta$ denote a smooth cutoff function defined by 
\[ 
\eta(x) = \begin{cases}  & 0 \quad \text{if } x \notin B_{g(0)}(q, 2r), \\
& \in [0,1] \quad \text{if } x \in B_{g(0)}(q, 2r) \setminus B_{g(0)}(q, r), \\
& 1  \quad \text{if } x \in B_{g(0)}(q, r).
\end{cases}
\]
Obviously $\eta$ satisfies $|\nabla \eta |_{g(0)} + | \nabla^2 \eta |_{g(0)} \leq C(r, g(0))$.
 By Lemma \ref{lem cut off funct deriv est} and (\ref{eq assum partial t derv g RM p}) we have
\[
\sup_{t \in [0,T ]} |\nabla \eta |_{g(t)} + | \nabla^2 \eta |_{g(t)} \leq 
 \Lambda (n, T, K, r, g(0)).
\]

Let $\beta_a =1$ and let $\beta_j, j=0,1, \cdots, a-1$, be nonnegative constants
 to be determined below. We define
\[
f_a(t) \doteqdot \sum^a_{j=0} \beta_{j} t^{j} \int_M  | \nabla^{2j}
\operatorname{Rm}(\cdot, t) |^2 \eta^{4j+2} d \mu .
\]
It follows from Proposition \ref{prop curv derv L2 evol ineq cutoff} 
with $\epsilon = \frac{1}{2}$ that
\begin{align*}
\frac{d}{dt} f_a(t) \leq & \sum^a_{j=1}\int_M  | \nabla^{2j}\operatorname{Rm}(
\cdot, t) |^2 \eta^{j+2} d \mu\cdot \left ( -  \frac{3}{2}  \beta_{j-1}t^{j-1}+ C_3 (j)
 \beta_{j} t^{j}  + j \beta_{j} t^{j -1} \right )  \\
& + C \left (1+  \| \operatorname{Rm} (\cdot, t) \|_{\infty, \eta >0}^2  +\|p(t)\|_{\infty, \eta >0} 
\right ) \| p \|^2_{2, \eta >0} 
\end{align*}
where 
\[
C_3(j) = C_1(j) \| \operatorname{Rm}(\cdot, t) \|_{2, 
\eta >0}^2 + C_2(j) \| p (t) \|^2_{2, \eta >0}
\]
and $C_1(j)$ and $C_2(j)$ are given in Proposition \ref{prop curv derv L2 evol ineq cutoff}.
 It is clear that by an appropriate inductive choice of the constants $\beta_i$ starting from  
 $\beta_a =1$ we obtain
\[
\frac{d}{dt}f_a(t)  \leq C \left (1+  \| \operatorname{Rm} (\cdot, t) \|_{\infty, \eta >0}^2 +
\|p(t)  \|_{\infty, \eta >0} \right )  \| p \|^2_{2, \eta >0}.
\]
Integrating this ODE  we get the required estimate for even $m =2a$.

For odd $m=2a-1$ the estimate follows from the interpolation inequality and  the even case.
 \end{proof}

\begin{remark} \label{rk about part t nabla g bound}
From the CBF equation (\ref{eq cbf eq p ellipt})  we know that $|\partial_t g| +
 |\nabla \partial_t g | $ in  (\ref{eq assum partial t derv g RM p}) can be bounded by assuming 
$$ \max_{a=0,1, \cdots, 3} \sup_{B_{g(0)}(q, 2r) \times [0,T ]} |\nabla^a
\operatorname{Rm} (\cdot, \cdot)| \leq K.
$$
\end{remark}

%%%%%%%%%%%%%%%%%%%%%%%%%%%
%%%%%%%%%%%%%%%%%%%%%%%%%%%
\subsection{ Shi's type pointwise-estimate}\label{sec appl Shi type est better}
Recall that in the proof of Theorem \ref{thm when extension CBF possible} 
we use the assumption of Sobolev constant's bound to get $C^0$-bound of the derivatives
 of curvatures. In Theorem \ref{thm ptwise Shi est for CBF-intro}  we remove the assumption. 
Now we give a proof of the theorem which is similar to that of  \cite[Theorem 4.4]{St13}.

\vskip .1cm
We define function 
\[
f_m(x, t, g) \doteqdot \sum^m_{ j =1} \left | \nabla^{j} \operatorname{Rm}(x,t) 
\right |^{ \frac{2}{2+j}} \quad \text{ on }M \times (0,T]. 
\]
We claim that for all sufficiently large $m \geq 3$ there is a constant $C = C(n, m)$ such that
\[
f_m(x, t, g) \leq C \left ( K +  t^{- \frac{1}{2}} \right ). 
\]
The claim implies the theorem and we will prove it by a contradiction argument.

Suppose the claim is false, then there is a sequence $\{(M_i^n , g_i (t), p_i(t))\}_{
t \in [0,T_i]}$  of complete  solutions to CBF (\ref{eq cbf eq p ellipt}) 
 with constant scalar curvature $s_{0i}$
 which satisfy the assumption of the theorem, together with points $(x_i , t_i )$  such that 
\[
\lim_{i \to \infty} \frac{f_m(x_i , t_i , g_i )}{K + t_i^{- \frac{1}{2}}} = \infty.
\]
Without loss of generality we may choose the points $(x_i , t_i )$ such that 
\begin{equation} \label{eq subsec 64 tem us 1}
\frac{f_m(x_i , t_i , g_i )}{K + t_i^{- \frac{1}{2}}} \geq \frac{1}{2} \sup_{M_i
 \times (0,T_i ]}  \frac{f_m(x, t, g_i )}{K + t^{- \frac{1}{2}}}. 
\end{equation}
Let constants
\[
\lambda_i \doteqdot f_m(x_i , t_i, g_i) \quad \text{ and } \quad  \tilde{s}_{0i} =
 \lambda_i^{-1} s_{0i}.
\]
We define the scaled metrics and potential functions by
\[
\tilde{g}_i (\tilde{t}) \doteqdot \lambda_i g \left ( t_i + \frac{\tilde{t}}{\lambda^2_i}\right ),  
\quad \tilde{p}_i( \tilde{t}) =\lambda^{-2}_i p_i \left ( t_i + \frac{\tilde{t}}{\lambda^2_i} \right ).
\]

By the scaling property of CBF $(\tilde{g}_i (\tilde{t}), \tilde{p}_i(\tilde{t}))$ satisfies CBF  
(\ref{eq cbf eq p ellipt}) with constant scalar curvature $\tilde{s}_{0i}$,
 and  exists on $\tilde{t} \in [-t_i \lambda^2_i , 0]$. 
 These solutions have the following simple properties.

\vskip .1cm
\noindent (P1) By the definition of $f_m$ we have
\begin{equation*}
f_m(x, \tilde{t}, \tilde{g}_i) = \lambda_i^{-1} f_m \left (x, t_i +
 \frac{\tilde{t}}{\lambda_i^2}, g_i \right ),
\end{equation*}
hence  $f_m(x_i , 0, \tilde{g}_i) = 1$.

\vskip .1cm
\noindent (P2)  Note that
\[
t^{\frac{1}{2}}_i \lambda_i = \frac{f_m(x_i , t_i, g_i )}{t_i^{-\frac{1}{2}}} \geq
 \frac{f_m(x_i , t_i, g_i )}{K + t_i^{-\frac{1}{2}}} \to \infty,
\]
we conclude that the solutions  $(\tilde{g}_i(\tilde{t}), \tilde{p}_i(\tilde{t}))$ exist on $[-1, 0]$ 
for $i$ sufficiently large.

\vskip .1cm 
\noindent (P3) 
By assumption (\ref{eq curv bdd improv ptwise shi assump}) we have that for any $(x,\tilde{t}) 
\in M_i \times [-1,0]$

$$\left | \operatorname{Rm}_{\tilde{g}_i} (x, \tilde{t}) \right 
| \leq\frac{K}{\lambda_i} \to 0  \quad \text{ as } i \to \infty.  $$

\vskip .1cm 
\noindent (P4) From (P1) we have that for $(\tilde{x} , \tilde{t} ) \in M_i \times [-1, 0]$
\begin{align*}
f_m(\tilde{x}, \tilde{t} , \tilde{g}_i)  = \frac{f_m \left (\tilde{x} ,t_i +\frac{\tilde{t}}
 {\lambda_i^2}, g_i \right )} {\lambda_i} = \frac{f_m \left (\tilde{x} ,t_i +
 \frac{\tilde{t}}{\lambda_i^2}, g_i \right )}{f_m(x_i ,t_i, g_i )} \leq  2 \cdot 
 \frac{K + \left (t_i + \frac{\tilde{t}}{\lambda_i^2} \right )^{-\frac{1}{2}}}{ K
 +t_i^{-\frac{1}{2}} } \leq 3,
\end{align*}
where we have used (\ref{eq subsec 64 tem us 1}) and (P2) to get the first inequality above.

\vskip .1cm
Let metric  $\tilde{h}_i (\tilde{t})$, defined on some ball $B(0, 2r) \subset
 \mathbb{R}^n$ with $\tilde{t} \in[-1,0]$, be a local lifting of $\tilde{g}_i (\tilde{t})$
  near $x_i$ by exponential map $\exp_{x_i}^{\tilde{g}_i (0)}$. By (P4) and
   (\ref{eq curv bdd improv ptwise shi assump}) we have a uniform bound of curvatures
\[
\sum_{j=0}^m| \nabla^{j}_{\tilde{h}_i (\tilde{t})} \operatorname{Rm}_{\tilde{h}_i
 (\tilde{t})} (\tilde{x})|_{\tilde{h}_i (\tilde{t})}  \leq C \quad  \text{ for } \tilde{x} \in B(0, 2r).
\]
This implies that metrics $\tilde{h}_i (\tilde{t})$ are uniformly equivalent to Euclidean 
metric on $B(0, 2r)$. Hence by Cheeger-Gromov  compactness theorem of Riemannian manifolds the
 sequence $\{ \tilde{h}_i (\tilde{t}) \}$ sub-converges to $\tilde{h}_{\infty}(\tilde{t})$
  in $C^{m-2-\alpha}$-topology with $\tilde{t} \in [-1,0]$. 
  We can improve the convergence of $\{\tilde{h}_i (\tilde{t}) \}$ to 
  $C^{\infty}$-convergence for $\tilde{t} \in [-\frac{1}{2},0]$ by
   using Theorem \ref{thm improved Shi curvature derivative conseq}. 
   Note that by (\ref{eq curv bdd improv ptwise shi assump}) and the scaling property
    we have uniform upper bounds for
\[
\sup_{(\tilde{x},\tilde{t}) \in B(0, 2r) \times [-1, 0]} | \tilde{p}_i(\tilde{x}, 
\tilde{t})|\quad \text{ and } \quad \| \tilde{p}_i ( \tilde{t} ) \|_{L^2(B(0,2r ))}.
\]
By Remark \ref{rk about part t nabla g bound} and (P4)  we may apply Theorem
\ref{thm improved Shi curvature derivative conseq} to $\tilde{h}_i (\tilde{t})$ on 
$B(0,2r ) \times [-1,0]$ to conclude that for each fixed $a \in \mathbb{N}$

\[
\int_{ B_{(0,r )}}| \nabla^a_{\tilde{h}_i (\tilde{t})} \operatorname{Rm}_{\tilde{h}_i 
(\tilde{t})}(\cdot) |^2 d \mu_{\tilde{h}_i (\tilde{t})}
\]
has a uniform upper bound.

Since metrics $\tilde{h}_i (0)$ are uniformly equivalent to the Euclidean metric
 on $B(0, r )$, it follows that the Sobolev constants for metric $\tilde{h}_i (0)$ on 
 $B(0, 2r )$ are uniformly bounded. As in the proof of Theorem \ref{thm when 
 extension CBF possible} it follows that the $C^a$-norms of the curvature of 
 $\tilde{h}_ i (0)$ on $B(0, \frac{1}{2}r )$ are uniformly bounded  for all $a$. 
 Thus by taking a further sub-sequence we conclude that $\tilde{h}_ i (0)$ is 
sub-convergent to $\tilde{h}_{\infty}(0)$ in  $C^\infty$,
 hence $f_m(0, 0, \tilde{h}_{\infty}) = 1$ from  (P1). 
 However, $\operatorname{Rm}_{\tilde{h}_{\infty}} = 0$ on $B(0, \frac{1}{2}r )$,
we get a contradiction and the theorem is proved.

%%%%%%%%%%%%%%%%%%%%%%%%%%%%%%
%%%%%%%%%%%%%%%%%%%%%%%%%%%%%%
%%%%%%%%%%%%%%%%%%%%%%%%%%%%%%
%%%%%%%%%%%%%%%%%%%%%%%%%%%%%%

\bibliographystyle{natbib}

\end{document}